\documentclass[11pt]{amsart}
\usepackage{amssymb,pb-diagram}
\usepackage{amsmath}
\usepackage{amsthm,amssymb,amsfonts}
\usepackage{amssymb,amscd}
\usepackage{epic,eepic,epsfig}
\usepackage{latexsym}
\usepackage{xy}
\xyoption{all}

\newtheorem{Thm}{Theorem}[section]

\newtheorem{Cor}[Thm]{Corollary}
\newtheorem{Lem}[Thm]{Lemma}
\newtheorem{Prop}[Thm]{Proposition}

\theoremstyle{definition}
\newtheorem{Def}[Thm]{Definition}

\theoremstyle{remark}

\def\ldots{\mathinner{\ldotp\ldotp\ldotp}}
\def\cdots{\mathinner{\cdotp\cdotp\cdotp}}

\def \cal{\mathcal}

\def \diam{\text{diam }}
\def\Lip{\text{Lip}}

\def\eps{\varepsilon}

\def\Mdb{\mathbb M}
\def\Ndb{\mathbb N}
\def\Qdb{\mathbb Q}
\def\Rdb{\mathbb R}

\def\Zdb{\mathbb Z}
\def\eps{\varepsilon}

\def\n{\overline n}
\def\m{\overline m}

\def\Lip{\text{Lip}}

\newcommand{\bib}{\bibitem}

\begin{document}

\title{The non-linear geometry of Banach spaces after Nigel Kalton}

\author{G. Godefroy}

\address{Institut de Mathématiques de Jussieu, 4 place Jussieu, 75005 Paris, France}
\email{godefroy@math.jussieu.fr}

\author{G. Lancien}

\address{Universit\'{e} de Franche-Comt\'{e}, Laboratoire de Math\'{e}matiques UMR 6623,
16 route de Gray, 25030 Besan\c{c}on Cedex, FRANCE.}
\email{gilles.lancien@univ-fcomte.fr}

\author{V. Zizler}

\address{Department of Mathematical Sciences, University of Alberta, Central Academic Building, T6G 2G1 Edmonton, Alberta, Canada}
\email{zizler@math.cas.cz}


\thanks{The second author was partially supported by the P.H.C. Barrande 26516YG}

\maketitle

\centerline{\it Dedicated to the memory of Nigel J. Kalton}

\begin{abstract}

This is a survey of some of the results which were obtained in the last twelve years on the non-linear geometry of Banach spaces. We focus on the contribution of the late Nigel Kalton.

\end{abstract}

\markboth{G. Godefroy, G. Lancien, V. Zizler}{The non-linear geometry of Banach spaces after Nigel Kalton}

\section{Introduction}

Four articles among Nigel Kalton's last papers are devoted to the non-linear geometry of Banach spaces (\cite{K3},
\cite{K4}, \cite{K5}, \cite{K6}). Needless to say, each of these works is important, for the results and also for the open problems it
contains. These articles followed a number of contributions due to Nigel Kalton (sometimes assisted by co-authors) which reshaped
the non-linear geometry of Banach spaces during the last decade. Most of these contributions took place after Benyamini-Lindenstrauss'
authoritative book \cite{BL} was released, and it seems that they are not yet accessible in a unified and organized way. The present
survey adresses this need, in order to facilitate the access to Kalton's results (and related ones) and to help trigger further research in this
widely open field of research. Nigel Kalton cannot be replaced, neither as a friend nor as the giant of mathematics he was. But his wish
certainly was that research should go on,  no matter what. This work is a modest attempt to fulfill this wish, and to honor his memory.

Let us outline the contents of this article. Section 2 gathers several tables, whose purpose is to present in a handy way what is known so far about the stability of several isomorphic classes under non-linear isomorphisms or emdeddings. We hope that this section will provide the reader with an easy access to the state of the art. Of course these tables contain a number of question marks, since our present knowledge is far from complete, even for classical Banach spaces. Section 3  displays several results illustrating the non-trivial fact that asymptotic structures are somewhat invariant under non-linear isomorphisms. Section 4 deals with embeddings of special graphs into Banach spaces, and the use of such embeddings for showing the stability of certain properties under non-linear isomorphisms. The non-separable theory is addressed in Section 5. Non-separable spaces behave quite differently from separable ones and this promptly yields to open problems (and even to undecidable ones in ZFC). Section 6 displays the link between coarse embeddings of discrete groups (more generally of locally finite metric spaces) into the Hilbert space or super-reflexive spaces, and the classification of manifolds up to homotopy equivalence. This section attempts to provide the reader with some feeling on what the Novikov conjecture is about, and some connections between the non-linear geometry of Banach spaces and the ``geometry of groups" in the sense of Gromov. Finally, the last section 7 is devoted to the Lipschitz-free spaces associated with a metric space, their use and their structure (or at least, what is known about it). Sections 2 and 6 contain no proof, but other sections do. These proofs were chosen in order to provide information on the tools we need.

This work is a survey, but it contains some statements (such as Theorem \ref{GP}, Theorem \ref{GKL} or the last Remark in section 5) which were not published  before. Each section contains a number of commented open questions. It is interesting to observe that much of these questions are fairly simple to state. Answering them could be less simple. Our survey demonstrates that non-linear geometry of Banach spaces is a meeting point for a variety of techniques, which have to join forces in order to allow progress. It is our hope that the present work will help stimulate such efforts. We should however make it clear that our outline of Nigel Kalton's last papers does not exhaust the content of these articles. We strongly advise the interested reader to  consult them for her/his own research.

\section{Tables}

This section consists of five tables: Table 1 lists a number of classical spaces and check when these Banach spaces are characterized by their metric or their uniform structure. Table 2 displays what is known about Lipschitz embeddings from a classical Banach space into another, and Table 3 does the same for uniform embeddings. Table 4 investigates the stability of certain isomorphism classes (relevant to a classical property) under Lipschitz or uniform homeomorphism. And finally, Table 5 does the same for non-linear embeddability.

References are given within the tables themselves, but in order to improve readability, we almost always used symbols (whose meaning is explained below) rather than using the numbering of the reference list. Questions marks mean of course that to the best of our knowledge, the corresponding question is still open.

Our notation for Banach spaces is standard. All Banach spaces will be real. From the recent textbooks that may be used in the area we mention \cite{AK0} and \cite{FHHMZ}.

\medskip

$\clubsuit$= Benyamini-Lindenstrauss book \cite{BL}.
\smallskip

$\spadesuit$= Kalton recent papers.
\smallskip

$\sharp$ = Mendel-Naor papers \cite{MN} and \cite{MN2}.
\smallskip


$\Delta$= Godefroy-Kalton-Lancien papers \cite{GKL} and \cite{GKL2001}.
\smallskip

$\square$= Godefroy-Kalton paper on free spaces \cite{GK}.
\smallskip

$\heartsuit$= Johnson-Lindenstrauss-Schechtman paper \cite{JLS}.
\smallskip

$\Diamond$= Textbook \cite{FHHMZ}.

\smallskip

$\diamond$= Basic linear theory or topology.

\smallskip

? = Unknown to the authors.

\vfill\eject

\begin{table}
\caption{Spaces determined by weaker structures} \vspace{0.5cm}
\begin{tabular}{|c|c|c|c|c|c|c|c|c|}
\hline
\hline
& Determined by & Determined by & \\
Space & its Lipschitz & its uniform & \\
& Structure & Structure & \\
\hline

$\ell_2$ &yes& yes&\\

        & & $\clubsuit$ &\\

\hline

$\ell_p$ &yes & yes &\\

$1<p<\infty$ & & & \\

$p\not=2$              & & $\clubsuit$  &     \\

\hline

$\ell_1$ &? & ? &\\

\hline

$c_0$ & yes & ? & \\

      &$\clubsuit$ &  &\\

\hline

$L_p$ & yes &? &\\

$1<p<\infty$ & & & \\

$p\not=2$  &$\clubsuit$ & &\\

\hline

$L_1$ & ? & ? &\\
\hline

$C[0, 1]$ & ? & ? &\\

\hline


$\ell_2(c)$ &yes&yes &  \\

                   & & $\clubsuit$& \\

\hline

$c_0(c)$ & no & no &\\

              &$\clubsuit$& &\\

\hline

$\ell_\infty$ &?& ? & \\


\hline

$\ell_p\oplus \ell_q$ &yes& yes &\\
$1<p<q<\infty$ & & &\\
$p, q\not= 2$  & &$\clubsuit$+$\spadesuit$ &\\

\hline

$\ell_p\oplus \ell_2$ &yes& ? &\\
$1<p<\infty$ & & & \\
$p\not= 2$   & $\clubsuit$& & \\

\hline

$J$ & yes & ?&\\
James' space    &$\clubsuit$+\cite{Cas}& &\\
\hline \hline

\end{tabular}

\end{table}

\begin{table}
\caption{Lipschitz Embeddings from the 1st column into the 1st
row} \vspace{0.5cm}
\begin{tabular}{|c|c|c|c|c|c|c|c|c|c|c|c|}
\hline
\hline
Space &$\ell_2$ & $\ell_q$ & $\ell_1$ & $c_0$ & $L_q$ & $L_1$ &  $C[0, 1]$ &$\ell_2(c)$ & $c_0(c)$ &$\ell_\infty$&\\
 & &$1<q<\infty$ & & &$1<q<\infty$ & & & & & & \\
      &  & $q\not= 2$& &           &$q\not= 2$&  & & & & & \\

\hline

$\ell_2$ &yes &no &no &yes &yes &yes &yes &yes &yes &yes & \\
         & &$\clubsuit$ &$\clubsuit$ &$\clubsuit$ &$\diamond$ &$\diamond$ &$\diamond$ &$\diamond$ & $\clubsuit$
&$\diamond$ & \\

\hline

$\ell_p$ &no &yes iff&no &yes &yes iff &yes iff &yes &no&yes &yes& \\
$1<p<\infty$              &   & $p=q$    &   &    &$q\le p<2$ &$p<2$ & & & &  & \\
         & & & & &or $p=q$ & & & & & & \\
$p\not= 2$           &$\clubsuit$   &$\clubsuit$  & $\clubsuit$ & $\clubsuit$&$\clubsuit$&$\clubsuit$ &$\diamond$ &$\clubsuit$ &$\clubsuit$ &$\diamond$ & \\

\hline

$\ell_1$&no &no &yes &yes &no &yes &yes &no &yes &yes& \\

       &$\clubsuit$  &$\clubsuit$& & $\clubsuit$& $\clubsuit$&$\diamond$ &$\diamond$ &$\clubsuit$ &$\clubsuit$&$\diamond$ & \\

\hline

$c_0$  &no  &no &no &yes &no &no &yes &no&yes &yes& \\

    &$\clubsuit$  &$\clubsuit$&$\clubsuit$ & &$\clubsuit$ &$\clubsuit$&$\diamond$ &$\clubsuit$
&$\diamond$ &$\diamond$ &\\

\hline

$L_p$ & no &no &no &yes &yes iff &yes iff &yes &no &yes &yes& \\
$1<p<\infty$           &    &   &   &    & $q\le p<2 $ & $p<2$              & & & & & \\
                & & & & &or $p=q$  & & & & & & \\

$p\not=2$ &$\clubsuit$ &$\clubsuit$ &$\clubsuit$ &$\clubsuit$ &$\clubsuit$ &$\clubsuit$ &$\diamond$ &$\clubsuit$ &$\clubsuit$ &$\diamond$ & \\

\hline

$L_1$ &no  &no &no &yes &no &yes &yes&no&yes &yes& \\

   &$\clubsuit$  &$\clubsuit$&$\clubsuit$&$\clubsuit$&$\clubsuit$ & &$\diamond$ &$\clubsuit$ &$\clubsuit$ &$\diamond$ & \\

\hline

$C[0, 1]$ &no  &no &no &yes &no &no &yes &no &yes &yes& \\

      &$\clubsuit$   &$\clubsuit$&$\clubsuit$&$\clubsuit$& $\clubsuit$&$\clubsuit$ & &$\clubsuit$ &$\clubsuit$ &$\diamond$ & \\

\hline

$\ell_2(c)$ &no &no &no &no &no &no &no &yes &? &yes & \\

&$\diamond$ &$\diamond$ &$\diamond$ &$\diamond$ &$\diamond$
&$\diamond$ &$\diamond$ & &
&$\spadesuit$ & \\

\hline

$c_0(c)$ &no &no &no &no &no &no &no&no &yes &yes& \\

              &$\diamond$  &$\diamond$ &$\diamond$ &$\diamond$ &$\diamond$ &$\diamond$ &$\diamond$ &$\clubsuit$& &$\spadesuit$ & \\

\hline

$\ell_\infty$&no  &no &no &no &no &no &no &no &? &yes & \\

              &$\diamond$ &$\diamond$ &$\diamond$ &$\diamond$& $\diamond$&$\diamond$ &$\diamond$ &$\clubsuit$ & && \\

\hline \hline

\end{tabular}

\end{table}



\begin{table}
\caption{Uniform Embeddings from the 1st column into the 1st row}
\vspace{0cm}
\begin{tabular}{|c|c|c|c|c|c|c|c|c|c|c|c|}
\hline
\hline
Space &$\ell_2$  & $\ell_q$& $\ell_1$ & $c_0$ & $L_q$ & $L_1$ &  $C[0, 1]$ & $\ell_2(c)$& $c_0(c)$ &$\ell_\infty$&\\
      &          &   $q\in (1,\infty)$& & & $q\in (1,\infty)$ & & & & & & \\

& &$q\not= 2$ & & &$q\not= 2$ & & & & & &\\

\hline

$\ell_2$ &yes &yes &yes &yes&yes &yes &yes &yes &yes&yes &\\

         & &$\clubsuit$ &$\clubsuit$ &$\clubsuit$ &$\diamond$ &$\diamond$ &$\diamond$ &$\diamond$ &$\clubsuit$ &$\diamond$&\\

\hline

$\ell_p$&yes iff &yes if&yes iff&yes &yes iff&yes iff &yes &yes iff &yes &yes & \\

$p\in (1,\infty)$       &$p<2$ &$p\le q$ &$p<2$ & &$p\le q$ &$p<2$ & &$p<2$ & & &  \\
 $p\not= 2$       & & or $p<2$ & & & or $p<2$ & & & & & &   \\
        & &\cite{Alb}+$\clubsuit$ & & & & & & & & & \\
        & & & & & & & & & & & \\

        &&no if & & & & & & & & &    \\
         & &  $p>2$ & & & & & & & & &  \\
         & & and  $q<p$ & & & & & & & & &  \\
      & $\clubsuit$  &$\sharp$   &$\clubsuit$ &$\clubsuit$ &$\sharp$ & $\clubsuit$&$\diamond$&$\clubsuit$&$\clubsuit$ &$\diamond$ & \\

\hline

$\ell_1$& yes&yes&yes &yes &yes&yes &yes &yes &yes &yes &\\

       &$\clubsuit$  & $\clubsuit$& &$\clubsuit$ & $\sharp$&$\diamond$ &$\diamond$ &$\clubsuit$ &$\clubsuit$
&$\diamond$ & \\

\hline

$c_0$&no  &no &no &yes&no &no &yes &no &yes &yes &\\

     &$\clubsuit$&$\spadesuit$ or $\sharp$ &$\clubsuit$ & &$\spadesuit$ or $\sharp$ &$\clubsuit$ & &$\clubsuit$ &$\diamond$ &$\diamond$ &\\

\hline

$L_p$ &yes iff &no if &yes iff&yes&yes iff &yes iff &yes&yes iff &yes &yes &\\
$p\in (1,\infty)$ &$p<2$ &$p>2$ & $p<2$& &$p\le q$&$p<2$ & &$p<2$ & & &   \\
                  & &and $q<p$ & & &or $p<2$ &  & & & & & \\

$p\not=2$  & &$\sharp$ & & & &  & & & & & \\

           & & & & & &  & & & & & \\
           & &yes if & & & &  & & & & & \\
           & &$p<2$ & & & &  & & & & & \\
           & &$\clubsuit$ & & & &  & & & & & \\
           & & & & & &  & & & & & \\
           & &? if & & & &  & & & & & \\
           &$\clubsuit$ &$2<p\le q$ &$\clubsuit$ &$\clubsuit$ &$\sharp$ &$\clubsuit$ &$\diamond$ &$\clubsuit$ &$\clubsuit$ &$\diamond$ &\\

\hline

$L_1$ & yes&yes &yes &yes&yes &yes &yes &yes &yes &yes &\\

     &$\clubsuit$&$\clubsuit$ &$\clubsuit$ &$\clubsuit$ &$\clubsuit$ & &$\diamond$ &$\clubsuit$ &$\clubsuit$ &$\diamond$ &\\

\hline

$C[0, 1]$&no &no &no & yes&no&no &yes &no &yes
 &yes &\\

         &$\clubsuit$&$\spadesuit$ or $\sharp$ &$\clubsuit$ &$\clubsuit$&$\spadesuit$ or $\sharp$ &$\clubsuit$ & &$\clubsuit$&$\clubsuit$&$\diamond$ &\\

\hline

$\ell_2(c)$ &no &no &no &no &no &no &no &yes&? &yes & \\
                   &$\diamond$ &$\diamond$ &$\diamond$ &$\diamond$ &$\diamond$ & &$\diamond$ & & &$\spadesuit$ & \\

\hline

$c_0(c)$ &no &no &no &no &no &no &no &no &yes&yes&  \\

                &$\diamond$ &$\diamond$ &$\diamond$ &$\diamond$ &$\diamond$ &$\diamond$ &$\diamond$ &$\clubsuit$& &$\spadesuit$ & \\

\hline

$\ell_\infty$ &no &no &no &no &no &no &no &no &?&yes &\\

              &$\diamond$ &$\diamond$ &$\diamond$ &$\diamond$ &$\diamond$ &$\diamond$ &$\diamond$ &$\clubsuit$ & & &  \\

\hline \hline

\end{tabular}

\vskip 1cm

A uniform embedding of a Banach space $X$ into a Banach space $Y$ is a uniform \hfill {}
homeomorphism  from $X$ onto a \underline{subset} of $Y$. Let us also mention that the same \hfill {}
table can be written about coarse embeddings (see the definition in section 3). \hfill {}
The set of references to be used is the same except for one paper by Nowak \cite{No}, \hfill {}
where it is proved that for any $p\in [1,\infty)$, $\ell_2$ coarsely embed into $\ell_p$.\hfill{}

\end{table}




\begin{table}
\caption{Stability under type of homeomorphism} \vspace{0.5cm}
\begin{tabular}{|c|c|c|c|}
\hline
\hline

Property & Lipschitz & uniform &\\

         &  & &\\

\hline

hilbertian & yes & yes & \\

           & &$\clubsuit$&\\

\hline

superreflexivity & yes & yes& \\
                 &  & $\clubsuit$&\\

\hline

reflexivity & yes & no &\\

            &$\clubsuit$& $\clubsuit$& \\

\hline

RNP & yes & no &\\

    &$\clubsuit$& $\clubsuit$&\\

\hline

Asplund & yes & no & \\

        &$\clubsuit$&$\clubsuit$&\\

\hline

containment of $\ell_1$ & ? & no& \\
                        & &$\clubsuit$ & \\

\hline

containment of $c_0$ & ? & ? & \\

\hline

BAP & yes & no &\\

    &$\square$&$\heartsuit$ &\\

\hline

Commuting BAP &?& no&\\
& & $\heartsuit$&\\

\hline

Existence of Schauder basis &? & ?&\\

\hline

Existence of M-basis & no & no &\\

                   &  $\Diamond$& & \\

\hline

Existence of unconditional basis &no& no&\\

                               &  $\Diamond$& & \\

\hline

renorming by Frechet smooth norm   & ? & no &\\

                                   &   & $\clubsuit$&\\

\hline

renorming by LUR norm & ? & ?&\\

\hline
renorming by UG norm & no& no &\\

                     &$\Diamond$& &\\

\hline

renorming by WUR norm &? &? &\\

\hline

renorming by AUS norm &yes &yes&\\

                     & &$\Delta$ &\\

\hline \hline

\end{tabular}
\end{table}

\begin{table}
\caption{Properties shared by embedded spaces} \vspace{0.5cm}
\begin{tabular}{|c|c|c|c|}
\hline
\hline

Property & Lipschitz & coarse-Lipschitz &\\

         &  & &\\

\hline

Hilbertian & yes & yes & \\

           &$\clubsuit$  &$\clubsuit$ &\\

\hline

superreflexivity & yes & yes & \\
                 &$\clubsuit$  &$\clubsuit$ &\\

\hline

reflexivity & yes& no &\\

            &$\clubsuit$& $\clubsuit$ & \\

\hline

RNP & yes & no &\\

    &$\clubsuit$&$\clubsuit$ &\\

\hline

Asplund & no & no & \\

        &$\clubsuit$&$\clubsuit$&\\

\hline

renorming by Frechet smooth norm   & no & no &\\

                                   &$\clubsuit$   & $\clubsuit$&\\

\hline

renorming by LUR norm & ? & ?&\\

\hline
renorming by UG norm & no& no &\\

                     &$\Diamond$& &\\

\hline

renorming by WUR norm &? &? &\\

\hline

renorming by AUS norm &no&no&\\
                      &$\clubsuit$& &\\

\hline

reflexive+renorming by AUS norm &yes& ?&\\
                                &$\clubsuit$& &\\

\hline
reflexive+renorming by AUS norm &yes &yes &\\
+renorming by AUC norm & &\cite{BKL} &\\

\hline \hline

\end{tabular}
\end{table}

We say that a Banach space $X$ is determined by its Lipschitz (respectively uniform) structure if a Banach space $Y$ is linearly isomorphic to $X$ whenever $Y$ is Lipschitz homeomorphic (respectively uniformly homeomorphic) to $X$.

${\ }$ \vfill\eject

A Lipschitz embedding of a Banach space $X$ into a Banach space $Y$ is a Lipschitz homeomorphism from $X$ onto a \underline{subset} (in general non linear) of $Y$.

${\ }$ \vfill\eject

${\ }$ \vfill\eject

More precisely, the question addressed is the following. If a Banach space $X$ Lipschitz or coarse-Lipschitz (see definition in section 3) embed into a Banach space $Y$ which has one the properties listed in the first column, does $X$ satisfy the same property?

${\ }$ \vfill\eject


\section{Uniform and asymptotic structures
of Banach spaces}

In this section we will study the stability of the uniform
asymptotic smoothness or uniform asymptotic convexity of a Banach
space under non linear maps such as uniform homeomorphisms and
coarse-Lipschitz embeddings.

\subsection{Notation - Introduction}

\begin{Def} Let $(M,d)$ and $(N,\delta)$ be two metric
spaces and $f:M\to N$ be a mapping.

\smallskip\noindent (a) $f$ is a {\it Lipschitz isomorphism} (or {\it Lipschitz homeomorphism}) if $f$ is a
bijection and $f$ and $f^{-1}$ are Lipschitz. We denote
$M \buildrel {Lip}\over {\sim} N$ and we say that $M$ and $N$ are {\it Lipschitz equivalent}.

\smallskip\noindent (b) $f$ is a {\it uniform homeomorphism} if $f$ is a
bijection and $f$ and $f^{-1}$ are uniformly continuous (we denote
$M \buildrel {UH}\over {\sim} N$).

\smallskip\noindent (c) If $(M,d)$ is unbounded, we define
$$\forall s>0,\ \ Lip_s(f)=\sup\{\frac{\delta((f(x),f(y))}{d(x,y)},\ d(x,y)\ge s\}\ \
{\rm and}\ \ Lip_\infty(f)=\inf_{s>0}Lip_s(f).$$

\noindent $f$ is said to be {\it coarse Lipschitz} if
$Lip_\infty(f)<\infty$.

\smallskip\noindent (d) $f$ is a {\it coarse Lipschitz embedding} if there exist
$A>0, B>0, \theta\geq 0$ such that

$$\forall x,y\in M \ \ d(x,y)\ge \theta \Rightarrow Ad(x,y)\le \delta(f(x),f(y))\le Bd(x,y).$$
We denote $M \buildrel {CL}\over {\hookrightarrow} N$.

More generally, $f$ is a {\it coarse embedding} if there exist two real-valued
functions $\rho_1$ and $\rho_2$ such that $\lim_{t\rightarrow+\infty}\rho_1(t)=+\infty$ and

$$\forall x,y\in M \ \ \rho_1(d(x,y))\le \delta(f(x),f(y))\le \rho_2(d(x,y)).$$

\smallskip\noindent (e) An $(a,b)$-net in the metric space $M$ is a subset $\cal M$ of $M$
such that for every $z\neq z'$ in $\cal M$,  $d(z,z')\ge a$ and for
every $x$ in $M$, $d(x,\cal M)< b$.

\noindent Then a subset $\cal M$ of $M$ is a {\it net} in $M$ if it is an $(a,b)$-net for some $0<a\le b$.

\smallskip\noindent (f) Note that two nets in the same infinite
dimensional Banach space are always Lipschitz equivalent (see
Proposition 10.22 in \cite{BL}).

\noindent Then two infinite dimensional Banach spaces $X$ and $Y$ are said to be {\it net
equivalent} and we denote $X \buildrel {N}\over {\sim} Y$, if
there exist a net $\cal M$ in $M$ and a net $\cal N$ in $N$ such
that $\cal M$ and $\cal N$ are Lipschitz equivalent.
\end{Def}

\noindent{\bf Remark.} It follows easily from the triangle
inequality that a uniformly continuous map defined on a Banach
space is coarse Lipschitz and a uniform homeomorphism between
Banach spaces is a bi-coarse Lipschitz bijection (see Proposition 1.11 in \cite{BL} for details). Therefore if $X$
and $Y$ are uniformly homeomorphic Banach spaces, then they are
net equivalent. It has been proved only recently by Kalton in
\cite{K4} that there exist two net equivalent Banach spaces that
are not uniformly homeomorphic. However the finite dimensional structures of Banach spaces are preserved under net equivalence (see Proposition 10.19 in \cite{BL}, or Theorem \ref{Ribelocal} below).

\medskip The main question addressed in this section is the
problem of the uniqueness of the uniform (or net) structure of a
given Banach space. In other words, whether $X \buildrel {UH}\over
{\sim} Y$ (or $X \buildrel {N}\over {\sim} Y$) implies that $X$ is linearly isomorphic to $Y$ (which we shall denote $X\simeq
Y$)? Even in the separable case, the general answer is negative.
Indeed Ribe \cite{R1} proved the following.

\begin{Thm}{\bf(Ribe 1984)} Let $(p_n)_{n=1}^\infty$ in $(1,+\infty)$ be a
strictly decreasing sequence such that $\lim p_n=1$. Denote
$X=(\sum_{n=1}^\infty L_{p_n})_{\ell_2}$. Then $ X\buildrel
{UH}\over {\sim} X\oplus L_1$.
\end{Thm}

Therefore reflexivity is not preserved under coarse-Lipschitz
embeddings or even uniform homemorphisms. On the other hand, Ribe
\cite{R2} proved that local properties of Banach spaces are
preserved under coarse-Lipschitz embeddings. More precisely.

\begin{Thm}\label{Ribelocal}{\bf(Ribe 1978)} Let $X$ and $Y$ be two Banach spaces
such that $X \buildrel {CL}\over {\hookrightarrow} Y$. Then there
exists a constant $K\ge 1$ such that for any finite dimensional
subspace $E$ of $X$ there is a finite dimensional subspace $F$ of
$Y$ which is $K$-isomorphic to $E$.
\end{Thm}

\noindent{\bf Remark.} If we combine this result with Kwapien's
theorem, we immediately obtain that a Banach space which is net equivalent
to $\ell_2$ is linearly isomorphic to $\ell_2$.

\medskip As announced, we will concentrate on some asymptotic
properties of Banach spaces. So let us give the relevant
definitions.

\begin{Def} Let $(X,\|\ \|)$ be a Banach space and
$t>0$. We denote by $B_X$ the closed unit ball of $X$ and by $S_X$
its unit sphere. For $x\in S_X$ and $Y$ a closed linear subspace
of $X$, we define
$$\overline{\rho}(t,x,Y)=\sup_{y\in S_Y}\|x+t y\|-1\ \ \ \ {\rm and}\ \ \
\ \overline{\delta}(t,x,Y)=\inf_{y\in S_Y}\|x+t y\|-1.$$ Then
$$\overline{\rho}_X(t)=\sup_{x\in S_X}\ \inf_{{\rm
dim}(X/Y)<\infty}\overline{\rho}(t,x,Y)\ \ \ \ {\rm and}\ \ \ \
\overline{\delta}_X(t)=\inf_{x\in S_X}\ \sup_{{\rm
dim}(X/Y)<\infty}\overline{\delta}(t,x,Y).$$ The norm $\|\ \|$ is said to be
{\it asymptotically uniformly smooth} (in short AUS) if
$$\lim_{t \to 0}\frac{\overline{\rho}_X(t)}{t}=0.$$
It is said to be {\it asymptotically uniformly convex} (in short
AUC) if
$$\forall t>0\ \ \ \ \overline{\delta}_X(t)>0.$$
\end{Def}

These moduli have been first introduced by Milman in \cite{M}. We also refer the reader to \cite{JLPS} and \cite{Du} for reviews on these.

\medskip\noindent {\bf Examples.}

\smallskip (1) If $X=(\sum_{n=1}^\infty F_n)_{\ell_p}$, $1\le p<\infty$ and the $F_n$'s are finite
dimensional, then
$\overline{\rho}_X(t)=\overline{\delta}_X(t)=(1+t^p)^{1/p}-1$. Actually, if a separable reflexive Banach space has equivalent norms with moduli of asymptotic convexity and smoothness of power type $p$, then it is isomorphic to a subspace of an $l_p$-sum of finite dimensional spaces \cite{JLPS}.

\smallskip (2) For all $t\in (0,1)$, $\overline{\rho}_{c_0}(t)=0$. And again, if $X$ is separable and
$\overline{\rho}_{X}(t_0)=0$ for some $t_0>0$ then $X$ is isomorphic to a subspace of $c_0$ \cite{GKL}.

\medskip We conclude this introduction by
mentioning the open questions that we will comment on in the course
of this section.

\smallskip\noindent {\bf Problem 1.} Let $1<p<\infty$ and $p\neq 2$. Does
$\ell_p\oplus \ell_2$ have a unique uniform or net structure? Does $L_p$ have a unique uniform or net structure?

\smallskip\noindent {\bf Problem 2.} Assume that $Y$ is a reflexive AUS Banach space and that $X$ is a
Banach space which coarse-Lipschitz embeds into $Y$. Does $X$
admit an equivalent AUS norm?

\smallskip\noindent {\bf Problem 3.}  Assume that $Y$ is an AUC Banach space and that $X$ is a Banach space which coarse-Lipschitz embeds into $Y$. Does $X$
admit an equivalent AUC norm?

\subsection{The approximate midpoints principle}

Given a metric space $X$, two points $x,y \in X$, and $\delta >0$,
the approximate metric midpoint set between $x$ and $y$ with error
$\delta$ is the set:
$$Mid(x,y,\delta)=\left \{z \in X: ~\max\{d(x,z), d(y,z)\} \leq (1+\delta)\frac{d(x,y)}{2} \right \}.$$

The use of approximate metric midpoints in the study of nonlinear
geometry is due to Enflo in an unpublished paper and has since
been used extensively,  e.g. \cite{Bo1}, \cite{Gor} and \cite{JLS}.

The following version of the approximate midpoint Lemma was
formulated in \cite{KR} (see also \cite{BL} Lemma 10.11).

\begin{Prop}\label{midpoint}  Let $X$ be a normed space and suppose $M$ is a metric space.
Let $f:X\to M$ be a coarse Lipschitz map.  If $Lip_\infty(f)>0$
then for any $t,\eps>0$ and any $0<\delta<1$ there exist $x,y\in
X$ with $\|x-y\|>t$ and
$$ f(\mathrm{Mid}(x,y,\delta))\subset \mathrm{Mid}(f(x),f(y),(1+\eps)\delta).$$
\end{Prop}

In view of this Proposition, it is natural to study the
approximate metric midpoints in $\ell_p$. This is done in the next
lemma, which is rather elementary and can be found in \cite{KR}.

\begin{Lem}\label{midpointellp} Let $1\le p<\infty$. We denote $(e_i)_{i=1}^\infty$ the canonical basis of
$\ell_p$ and for $N\in \Ndb$, let $E_N$ be the closed linear span of
$\{e_i,\ i>N\}$. Let now $x,y\in \ell_p$, $\delta\in (0,1)$,
$u=\frac{x+y}{2}$ and $v=\frac{x-y}{2}$. Then

\smallskip\noindent (i) There exists $N\in \Ndb$ such that $u+\delta^{1/p}\|v\|B_{E_N} \subset Mid(x,y,\delta).$

\smallskip\noindent (ii) There is a compact subset $K$ of $\ell_p$ such that $Mid(x,y,\delta)\subset
K+2\delta^{1/p}\|v\|B_{\ell_p}$.
\end{Lem}

We can now combine Proposition \ref{midpoint} and Lemma
\ref{midpointellp} to obtain



\begin{Cor}\label{midpointellqellp} Let $1\le p<q<\infty$.

Then $\ell_q$ does not coarse Lipschitz embed into $\ell_p$.
\end{Cor}


\noindent {\bf Remark.} This statement can be found in \cite {KR} but was implicit in \cite{JLS}. It already indicates that,
because of the approximate midpoint principle, some uniform
asymptotic convexity has to be preserved under coarse Lipschitz
embeddings. This idea will be pushed much further in section
\ref{uac}.

\subsection{Gorelik principle and applications}
Our goal is now to study the stability of the uniform asymptotic
smoothness under non linear maps. The first tool that we shall
describe is the Gorelik principle. It was initially devised by
Gorelik in \cite{Gor} to prove that $\ell_p$ is not uniformly
homeomorphic to $L_p$, for $1<p<\infty$. Then it was developed by
Johnson, Lindenstrauss and Schechtman \cite{JLS} to prove that for
$1<p<\infty$, $\ell_p$ has a unique uniform structure. It is
important to underline the fact that the Gorelik principle is only
valid for certain bijections. In fact, the uniqueness of the
uniform structure of $\ell_p$ can be proved without the Gorelik
principle, by using results on the embeddability of special metric graphs as we
shall see in section \ref{uas}. Nevertheless, some other results still need the use of the Gorelik principle. This principle is
usually stated for homeomorphisms with uniformly continuous inverse (see Theorem 10.12 in \cite{BL}). Although it is probably known, we have
not found its version for net equivalences. Note that a Gorelik principle is proved in \cite{BL} (Proposition 10.20) for net equivalences between a Banach space and $\ell_p$. So we will describe here how to obtain a general statement.

\begin{Thm}\label{GP} {(Gorelik Principle.)} Let $X$ and $Y$ be two  Banach spaces. Let $X_0$ be a closed linear subspace
of $X$ of finite codimension. If $X$ and $Y$ are net equivalent, then there are continuous maps $U:X\to Y$ and $V:Y\to X$,
and constants $K,C,\alpha_0> 0$ such that:
$$\forall x\in X\ \ \|VUx-x\|\le C\ \ {\rm and}\ \ \forall y\in Y\
\ \|UVy-y\|\le C\ \ \ {\rm and}$$ for all $\alpha>\alpha_0$ there
is a compact subset $M$ of $Y$ so that
$$\frac{\alpha}{16K}B_Y \subset M+CB_Y+U(\alpha B_{X_0}).$$
\end{Thm}

\begin{proof} Suppose that $\cal N$ is a net of the Banach space $X$, that
$\cal M$ is a net of the Banach space $Y$ and that $\cal N$ and
$\cal M$ are Lipschitz equivalent. We will assume as we may that
$\cal N$ and $\cal M$ are $(1,\lambda)$-nets for some $\lambda>1$
and that $\cal N= (x_i)_{i\in I}$, $\cal M=(y_i)_{i\in I}$ with
$$\forall i,j\in I\ \ \ K^{-1}\|x_i-x_j\|\le \|y_i-y_j\|\le
K\|x_i-x_j\|,$$  for some $K\ge 1$. Let us denote by $B_E(x,\lambda)$ the open ball of center $x$ and radius $\lambda$ in the Banach space $E$. Then we can find a continuous partition of unity $(f_i)_{i\in I}$ subordinate to $(B_X(x_i,\lambda))_{i\in I}$ and a
continuous partition of unity $(g_i)_{i\in I}$ subordinate to
$(B_Y(y_i,\lambda))_{i\in I}$. Now we set:
$$Ux=\sum_{i\in I} f_i(x)y_i,\ x\in X\ \ {\rm and}\ \
Vy=\sum_{i\in I} g_i(y)x_i,\ y\in Y.$$
The maps $U$ and $V$ are clearly continuous. We shall now state and prove two lemmas about them.

\begin{Lem}\label{L1} (i) Let $x\in X$ be such that $\|x-x_i\|\le r$,
then $\|Ux-y_i\|\le K(\lambda+r)$.

(ii) Let $y\in Y$ be such that $\|y-y_i\|\le r$, then $\|Vy-x_i\|\le
K(\lambda+r)$.
\end{Lem}

\begin{proof} We will only prove (i). If $f_j(x)\neq 0$, then $\|x-x_j\|\le\lambda$. So
$\|x_i-x_j\|\le \lambda+r$ and $\|y_i-y_j\|\le K(\lambda+r)$. We
finish the proof by writing
$$Ux-y_i=\sum_{j,f_j(x)\neq 0} f_j(x)(y_j-y_i).$$
\end{proof}

\begin{Lem}\label{L2} Let $C=(1+K+2K^2)\lambda$. Then
$$\forall x\in X\ \ \|VUx-x\|\le C\ \ {\rm and}\ \ \forall y\in Y\
\ \|UVy-y\|\le C.$$
\end{Lem}

\begin{proof} We only need to prove one inequality. So let $x\in X$ and
pick $i\in I$ such that $\|x-x_i\|\le \lambda$. By the previous
lemma, we have $\|Ux-y_i\|\le 2K\lambda$ and $\|VUx-x_i\|\le
\lambda(K+2K^2)$. Thus $\|VUx-x\|\le (1+K+2K^2)\lambda$.
\end{proof}

We now recall the crucial ingredient in the proof of Gorelik Principle (see step (i) in the proof of Theorem 10.12 in \cite{BL}). This statement relies on Brouwer's fixed point
theorem and on the existence of Bartle-Graves continuous selectors. We refer the reader to \cite{BL} for its proof.

\begin{Prop}\label{G} Let $X_0$ be a finite-codimensional subspace
of $X$. Then, for any $\alpha >0$ there is a compact subset $A$ of
$\frac{\alpha}{2}B_X$ such that for every continuous map
$\phi:A\to X$ satisfying $\|\phi(a)-a\|\le \frac{\alpha}{4}$ for
all $a\in A$, we have that $\phi(A)\cap X_0\neq \emptyset$.
\end{Prop}

We are now ready to finish the proof of Theorem \ref{GP}. Fix $\alpha>0$ such that $\alpha>\max\{8C,96K\lambda\}$ and
$y\in \frac{\alpha}{16K}B_Y$ and define $\phi:A\to X$ by
$\phi(a)=V(y+Ua)$. The map $\phi$ is clearly continuous and we have that for all $a\in A$:
$$\|\phi(a)-a\|\le \|V(y+Ua)-VUa\|+\|VUa-a\|\le
\frac{\alpha}{8}+\|V(y+Ua)-VUa\|.$$ Now, pick $i$ so that
$\|Ua-y_i\|\le\lambda$ and $j$ so that $\|y+Ua-y_j\|\le \lambda$.
Then $\|VUa-x_i\|\le 2K\lambda$ and $\|V(y+ua)-x_j\|\le 2K\lambda$.
But
$$\|x_i-x_j\|\le K\|y_i-y_j\|\le K\|y_i-Ua\|+K\|Ua+y-y_j\|+K\|y\|\le
(2\lambda+\|y\|)K.$$ So
$$\|V(y+Ua)-VUa\|\le
6K\lambda+K\|y\|\le 6K\lambda+\frac{\alpha}{16}\le\frac{\alpha}{8}.$$ Thus
$\|\phi(a)-a\|\le \frac{\alpha}{4}$.

So it follows from Proposition \ref{G} that there exists $a\in A$
such that $\phi(a)\in X_0$. Besides, $\|a\|\le \frac{\alpha}{2}$
and $\|\phi(a)-a\|\le \frac{\alpha}{4}$, so $\phi(a)=V(y+Ua)\in
\alpha B_{X_0}$. But we have that $\|UV(y+Ua)-(y+Ua)\|\le C$. So
if we consider the compact set $M=-U(A)$, we have that $y\in
M+CB_Y+U(\alpha B_{X_0})$. This finishes the proof of Theorem \ref{GP}.

\end{proof}

We can now apply the above Gorelik principle to obtain the net
version of a result appeared in \cite{GKL2001}. This result is new.

\begin{Thm}\label{GKL} Let $X$ and $Y$ be Banach spaces. Assume that $X$ is net equivalent to $Y$ and that $X$ is AUS. Then $Y$
admits an equivalent AUS norm. More precisely, if $\overline
\rho_X(t)\le Ct^p$ for $C>0$ and $p\in (1,\infty)$, then, for any $\eps
>0$, $Y$ admits an equivalent norm $\|\ \|_\eps$ so that
$\overline \rho_{\|\ \|_\eps}(t)\le C_\eps t^{p-\eps}$ for some
$C_\eps>0$.
\end{Thm}

\noindent  The proof is actually done by constructing a sequence
of dual norms as follows:
$$\forall y^*\in Y^*\ \ \ |y^*|_k=\sup\Big{\{}\frac{\langle y^*,Ux-Ux'\rangle}{\|x-x'\|},\
\|x-x'\|\ge 4^k\Big{\}}.$$ For $k$ large enough they are all equivalent. Then for $N$ large enough
the predual norm of the norm defined by $$\forall y^*\in Y^*\ \
\|y^*\|_N=\frac1N \sum_{k=k_0+1}^{k=k_0+N}|y^*|_k,$$ is the dual of an equivalent AUS norm on $Y$ with the desired modulus of asymptotic smoothness. The proof follows the lines of the argument given in \cite{GKL2001} but uses the above version of Gorelik principle.

\medskip\noindent {\bf Remarks.}

(1) It must be pointed out that the quantitative estimate in the
above result is optimal as it follows from a remarkable example
obtained by Kalton in \cite{K5}.

\smallskip

(2) If the Banach spaces $X$ and $Y$ are Lipschitz equivalent and $X$ is AUS then the norm defined on $Y^*$ by:
$$|y^*|=\sup\Big{\{}\frac{\langle y^*,Ux-Ux'\rangle}{\|x-x'\|},\ x\neq x'\Big{\}}$$
is a dual norm of a norm $|\ |$ on $Y$ such that $\overline \rho_{|\ |}(t)\le c\overline \rho_X(ct)$ for some $c>0$. When $X$ is a subspace of $c_0$, this implies that $Y$ is isomorphic to a subspace of $c_0$. Finally, when $X=c_0$, one gets that $X$ is isomorphic to $c_0$ (see \cite{GKL}).

\smallskip

(3) Other results were originally derived from the
Gorelik principle. We have chosen to present them in the next
subsections as consequences of more recent and possibly more
intuitive graph techniques introduced by Kalton and Randrianarivony
in \cite{KR} and later developed by Kalton in \cite{K6}.

\subsection{Uniform asymptotic smoothness and Kalton-Randrianarivony's
graphs}\label{uas}

The fundamental result of this section is about the minimal
distortion of some special metric graphs into a reflexive and
asymptotically uniformly smooth Banach space. These graphs have
been introduced by Kalton and Randrianarivony in \cite{KR}
and are defined as follows.
\smallskip Let $\Mdb$ be an infinite subset of $\Ndb$ and $k\in \Ndb$ and fix
$a=(a_1,..,a_k)$ a sequence of non zero real numbers. We denote
$$G_k(\Mdb)=\{\n=(n_1,..,n_k),\ n_i\in\Mdb\ \ n_1<..<n_k\}.$$ Then we equip $G_k(\Mdb)$
with the distance
$$\forall \n,\m \in G_k(\Mdb),\ \ d_a(\n,\m)=\sum_{j,\ n_j\neq m_j} |a_j|.$$

Note also that it is easily checked that $\overline \rho _Y$ is an
Orlicz function. Then,  we define the Orlicz sequence space:
$$\ell_{\overline \rho _Y}=\{a\in \Rdb^\Ndb,\ \exists r>0\ \
\sum_{n=1}^\infty \overline \rho _Y\Big(\frac{|a_n|}{r}\Big)<\infty\},$$
equipped with the Luxemburg norm
$$\|a\|_{\overline \rho _Y}=\inf\{r>0,\ \sum_{n=1}^\infty \overline \rho _Y\Big(\frac{|a_n|}{r}\Big)\le 1\}.$$

\begin{Thm}\label{KaRa2}{\bf (Kalton-Randrianarivony 2008)} Let $Y$ be a reflexive
Banach space, $\Mdb$ an infinite subset of $\Ndb$ and
$f:(G_k(\Mdb),d_a)\to Y$ a Lipschitz map. Then for any $\eps>0$,
there exists an infinite subset $\Mdb'$ of $\Mdb$ such that:
$$\diam f(G_k(\Mdb')) \le 2eLip(f)\|a\|_{\overline \rho _Y} +\eps.$$
\end{Thm}

\medskip The proof is done by induction on $k$ and uses iterated weak
limits of subsequences and a Ramsey argument. Such techniques will
be displayed in the next two sections.

\medskip\noindent{\bf Remark.} The reflexivity assumption is
important. Indeed, by Aharoni's Theorem the spaces
$(G_k(\Ndb),d_a)$ Lipschitz embed into $c_0$ with a distortion
controlled by a uniform constant. But $\|a\|_{\overline \rho
_{c_0}}=\|a\|_\infty$, while $\diam G_k(\Mdb')=\|a\|_1$.

\medskip As it is described in \cite{K6} one can deduce the following.

\begin{Cor}\label{KaRa3} Let $X$ be a Banach space and $Y$ be a reflexive Banach space. Assume that $X$
coarse Lipschitz embeds into $Y$. Then there exists $C>0$ such
that for any normalized weakly null sequence $(x_n)_{n=1}^\infty$
in $X$ and any sequence $a=(a_1,..,a_k)$ of non zero real numbers,
there is an infinite subset $\Mdb$ of $\Ndb$ such that:
$$\|\sum_{i=1}^k a_ix_{n_i}\| \le C
\|a\|_{\overline \rho_Y},\ \ {\text for\ every}\ \n \in G_k(\Mdb).$$
\end{Cor}

\begin{proof} The result is obtained by applying Theorem
\ref{KaRa2} to $f=g\circ h$, where $g$ is a coarse-Lipschitz
embedding from $X$ into $Y$ and $h:(G_k(\Ndb),d_a)\to X$ is
defined by $h(\n)=\lambda \sum_{i=1}^k a_ix_{n_i}$ for some large
enough $\lambda>0$.
\end{proof}

In fact, this is stated in \cite{K6} in the following more
abstract way.

\begin{Cor} Let $X$ be a Banach space and $Y$ be a reflexive Banach space. Assume that $X$
coarse Lipschitz embeds into $Y$. Then there exists $C>0$ such
that for any spreading model $(e_i)_i$ of a normalized weakly null
sequence in $X$ (whose norm is denoted $\|\ \|_S$) and any finitely supported sequence $a=(a_i)$ in
$\Rdb$:
$$\|\sum a_i e_i\|_S\le C\|a\|_{\overline \rho_Y}.$$
\end{Cor}

\subsection{Applications}

The first consequence is the following.

\begin{Cor}\label{graphellqellp} Let $1\le q\neq p<\infty.$

Then $\ell_q$ does not coarse Lipschitz embed into $\ell_p$.
\end{Cor}

\begin{proof} If $q<p$, this follows immediately from the previous
results.

\noindent If $q>p$, this is Corollary \ref{midpointellqellp}.
\end{proof}

Then we can deduce the following result, proven in \cite{JLS} under the
assumption of uniform equivalence.

\begin{Thm}\label{JLS96} {\bf(Johnson, Lindenstrauss and Schechtman 1996)}

\noindent Let $1<p<\infty$ and $X$ a Banach space such that
$X\buildrel {N}\over {\sim} \ell_p$. Then $X\simeq \ell_p$.
\end{Thm}

\begin{proof}  Suppose that $X\buildrel {N}\over {\sim} \ell_p$, with $1<p<\infty$. We
may assume that $p\neq 2$. Then the ultra-products $X_{\cal U}$
and $(\ell_p)_{\cal U}$ are Lipschitz isomorphic and it follows
from the classical Lipschitz theory that $X$ is isomorphic to a
complemented subspace of $L_p=L_p([0,1])$. Now, it follows from
Corollary  \ref{graphellqellp} that $X$ does not contain any
isomorphic copy of $\ell_2$. Then we can conclude with a classical
result of Johnson and Odell \cite{JO} which asserts that any
infinite dimensional complemented subspace of $L_p$ that does not
contain any isomorphic copy of $\ell_2$ is isomorphic to $\ell_p$.
\end{proof}

\noindent{\bf Remark.} These linear arguments are taken from
\cite{JLS}. Note that the key step was to show that that $X$ does
not contain any isomorphic copy of $\ell_2$. In the original paper
\cite{JLS} this relied on the Gorelik principle. We have chosen to
present here a proof using this graph argument. In fact, more can
be deduced from this technique.

\begin{Cor} Let $1\le p<q<\infty.$ and $r\ge 1$ such that $r\notin \{p,q\}$.

Then $\ell_r$ does not coarse Lipschitz embed into $\ell_p\oplus
\ell_q$.
\end{Cor}

\begin{proof} When $r>q$, the argument is based on a midpoint technique. If $r<p$, it follows
immediately from Corollary \ref{KaRa3}. So we assume now that
$1\le p<r<q<\infty$ and $f=(g,h):\ell_r \to \ell_p\oplus_\infty
\ell_q$ is a coarse-Lipschitz embedding. Applying the midpoint
technique to the coarse Lipschitz map $g$ and then Theorem
\ref{KaRa2} to the map $h\circ \varphi$ with $\varphi$ of the form
$\varphi(\n)=u+\tau k^{-1/r}(e_{n_1}+..+e_{n_k})$, where $(e_n)$
is the canonical basis of $\ell_r$ and $\tau >0$ is large enough, leads to a contradiction.
\end{proof}

We can now state and prove the main result of \cite{KR}.

\begin{Thm} If $1<p_1<..<p_n<\infty$ are all different from 2, then
$\ell_{p_1}\oplus ...\oplus \ell_{p_n}$ has a unique net
structure.
\end{Thm}

\begin{proof} We will only sketch the proof for $\ell_p\oplus \ell_q$,
with $1<p<q<\infty$ such that $2\notin\{p,q\}$. Assume that $X$ is
Banach space such that $X\buildrel {N}\over {\sim} \ell_p\oplus
\ell_q$. The key point is again to show that $X$ does not contain
any isomorphic copy of $\ell_2$. This follows clearly from the
above corollary. To conclude the proof, we need to use a few deep
linear results. The cases $1<p<q<2$ and $2<p<q$, were actually
settled in \cite{JLS} for uniform homeomorphisms. So let us only
explain the case $1<p<2<q$. As in the proof of Theorem
\ref{JLS96}, we obtain that $X\subseteq_c L_p\oplus L_q$. Since
$\ell_2 \varsubsetneq X$ and $q>2$, a theorem of Johnson
\cite{J} insures that any bounded operator from $X$ into $L_q$
factors through $\ell_q$ and therefore that $X\subseteq_c
L_p\oplus \ell_q$. Then we notice that $L_p$ and $\ell_q$ are
totally incomparable, which means that they have no isomorphic
infinite dimensional subspaces. We can now use a theorem of
{\`{E}}del\v{s}te{\u\i}n and  Wojtaszczyk \cite{EW} to obtain that
$X\simeq F\oplus G$, with $F\subseteq_c L_p$ and $G\subseteq_c
\ell_q$. First it follows from \cite{P} that $G$ is isomorphic to
$\ell_q$ or is finite dimensional. On the other hand, we know that
$\ell_2 \varsubsetneq F$, and by the Johnson-Odell theorem
\cite{JO} $F$ is isomorphic to $\ell_p$ or finite dimensional.
Summarizing, we have that $X$ is isomorphic to $\ell_p$, $\ell_q$
or $\ell_p\oplus \ell_q$. But we already know that $\ell_p$ and
$\ell_q$ have unique net structure. Therefore $X$ is isomorphic to
$\ell_p\oplus \ell_q$.
\end{proof}

\noindent {\bf Remark.} Let $1<p<\infty$ and $p\neq 2$. It is
clear that the above proof cannot work for $\ell_p\oplus \ell_2$.
As we already mentioned, it is unknown whether
$\ell_p\oplus\ell_2$ has a unique uniform structure. The same
question is open for $L_p$, $1<p<\infty$ (see Problem 1).

\medskip However, let us indicate a few other results from \cite{KR}
that can be derived from Theorem \ref{KaRa2}. The following theorem is related to a
recent result of \cite{HOS} stating that if $2<p<\infty$, a subspace $X$ of $L_p$ which is not
isomorphic to a subspace of $\ell_p\oplus \ell_2$ contains an isomorphic copy of
$\ell_p(\ell_2)$.

\begin{Thm} Let $1<p<\infty$ and $p\neq 2$.

Then $\ell_p(\ell_2)$ and therefore $L_p$ do not coarse Lipschitz
embed into $\ell_p\oplus \ell_2$.
\end{Thm}

It follows from Ribe's counterexample that reflexivity is not
preserved under uniform homeomorphisms. However, the following is
proved in \cite{BKL}.

\begin{Thm}\label{reflexive} Let $X$ be a Banach space and $Y$ be a reflexive Banach space with an
equivalent AUS norm. Assume that $X$ coarse Lipschitz embeds into
$Y$. Then $X$ is reflexive.
\end{Thm}
The proof has three ingredients: a result of Odell and
Schlumprecht \cite{OS} asserting that $Y$ can be renormed in such
a way that $\overline{\rho}_Y \le \overline{\rho}_{\ell_p}$,
James' characterization of reflexivity and Theorem \ref{KaRa2}.

\medskip\noindent {\bf Remark.} Note that Theorems \ref{reflexive} and \ref{KaRa2} seem to take us very close to the solution of Problem 2. See also Corollary \ref{beta} below.

\subsection{Uniform asymptotic convexity}\label{uac}

Until very recently, there has been no corresponding result about
the stability of convexity. The only thing that could be mentioned
was the elementary use of the approximate midpoint principle that
we already described. In a recent article \cite{K6}, Nigel Kalton
made a real breakthrough in this direction. Let us first state his
general result.

\begin{Thm} Suppose $X$ and $Y$ are Banach spaces and that there
is a coarse Lipschitz embedding of $X$ into $Y$. Then there is
 a constant $C>0$ such that for any spreading model $(e_k)_{k=1}^\infty$ of a weakly null normalized sequence in $X$ (whose norm will be denoted $\|\ \|_S$), we have:
 $$\|e_1+..+e_k\|_{\overline \delta_Y}\le C\|e_1+..+e_k\|_S.$$
\end{Thm}

We will not prove this in detail. We have chosen instead to show
an intermediate result whose proof contains one of the key
ingredients. Let us first describe the main idea. We wish to use
the approximate midpoints principle. But, unless the space is very
simple or concrete (like $\ell_p$ spaces), the approximate
midpoint set is difficult to describe. Kalton's strategy in
\cite{K6} was, in order to prove the desired inequality, to define
an adapted norm on an Orlicz space associated with any given weakly null
sequence in $X$. Then, by composition, he was able to
reduce the question to the study of a coarse-Lipschitz map from that
space to the Orlicz space associated with the modulus
of asymptotic convexity of $Y$. Finally, as in $\ell_p$, the
approximate midpoints are not so difficult to study in an Orlicz
space. Before stating the result, we need some preliminary
notation.

\medskip Let $\phi$ be an Orlicz Lipschitz function. Then,
$N_\phi(1,t)=1+\phi(|t|)$ (for $t\in \Rdb$) extends to a norm
$N_\phi^2$ defined on $\Rdb^2$. We now define inductively a norm
on $\Rdb^j$ by
$N_\phi^j(x_1,..,x_j)=N_\phi^2(N_\phi^{j-1}(x_1,..,x_{j-1}),x_j)$
for $j\ge 2$. These norms are compatible and define a norm
$\Lambda_\phi$ on $c_{00}$. One can check that $\frac12\|\
\|_\phi\le \|\ \|_{\Lambda_\phi}\le \|\ \|_\phi.$

We also need to introduce the following quantities:
$$ \hat\delta_X(t)=\inf_{x\in\partial B_X}\sup_{E}\inf_{y\in\partial B_E}
\{\frac12(\|x+ty\|+\|x-ty\|)-1\}$$ where again $E$ runs through
all closed subspaces of $X$ of finite codimension. The function $\hat\delta_X(t)/t$ is increasing
and so $\hat\delta_X$ is equivalent to the convex function
$$\tilde\delta_X(t)=\int_0^t \frac{\hat\delta_X(s)}{s}\,ds.$$

\begin{Thm} Let $(\eps_i)_{i=1}^\infty$ be a sequence of
independent Rademacher variables on a probability space $\Omega$.
Assume that $X$ and $Y$ are Banach spaces and that there is a coarse
Lipschitz embedding of $X$ into $Y$. Then, there is a constant $c>0$
 such that given any $(x_n)$ weakly
null normalized $\theta$-separated sequence in $X$
and any integer $k$, there exist $n_1<..<n_k$ so that:
$$c\theta\|e_1+\cdots+e_k\|_{\tilde\delta_Y}\le \|\eps_1x_{n_1}+\cdots+\eps_kx_{n_k}\|_{L^1(\Omega,X)},$$
where $(e_k)_{k=1}^\infty$ is the canonical basis of $c_{00}$.
\end{Thm}

\begin{proof} For $k\in \Ndb$, let
$\sigma_k=\sup\left\{\|\sum_{j=1}^k\eps_jx_{n_j}\|_{L^1(\Omega,X)},\
\  n_1<n_2<\cdots<n_k\right\}.$ For each $k,$ define the Orlicz
function $F_k$ by
\begin{equation}\label{F_k}F_k(t)=\begin{cases} \sigma_kt/k, \qquad 0\le t\le 1/\sigma_k\\
t+1/k-1/\sigma_k,\qquad 1/\sigma_k\le
t<\infty.\end{cases}\end{equation} We introduce an operator
$T:c_{00}\to L_1(\Omega; X)$ defined by $T(\xi)
=\sum_{j=1}^{\infty}\xi_j\eps_j\otimes x_j.$  We omit the proof of
the fact that for all $\xi \in c_{00}$: $ \|T\xi\|\le
2\|\xi\|_{F_k}\le 4\|\xi\|_{\Lambda_{F_k}}.$

\noindent Assume now that $f:X\to Y$ is a map such that $f(0)=0$
and
$$ \|x-z\|-1\le \|f(x)-f(z)\|\le K\|x-z\|+1, \qquad x,z\in X.$$
We then define $g:(c_{00},\Lambda_{F_k})\to L_1(\Omega;Y)$ by
$g(\xi)=f\circ T\xi.$ It can be easily checked that $g$ is
coarse-Lipschitz and that $Lip_\infty(g)>0$. So we can apply the
approximate midpoint principle to the map $g$ and obtain that for
$\tau$ as large as we wish, there exist $\eta,\ \zeta\in c_{00}$
with $\|\eta-\zeta\|_{\Lambda_{N_k}}=2\tau$ such that
$$g(Mid(\eta,\zeta,1/k))\subset Mid(g(\eta),g(\zeta), 2/k).$$ Let
$m\in\mathbb N$ so that $\eta,\zeta\in span\,\{e_1,\ldots,e_{m-1}\}.$
It follows from the definition of $\Lambda_{N_k}$ that for $j\ge
m$: $\xi+\tau\sigma_k^{-1} e_j\in \mathrm{Mid}(\eta,\zeta,1/k)$,
where $\xi=\frac12(\eta+\zeta)$.

\noindent Thus the functions
$$h_j=f\circ(\sum_{i=1}^{m-1}\xi_i\eps_i\otimes x_i + \tau\sigma_k^{-1}\eps_j\otimes x_j),\ \ \ j\ge m$$
all belong to $Mid(g(\eta),g(\zeta), 2/k)$.  Since the $\eps_i$'s
are independent so do the functions
$$h'_j=f(\sum_{i=1}^{m-1}\xi_i\eps_i\otimes x_i + \tau\sigma_k^{-1}\eps_m\otimes x_j),\ \ \ j\ge m.$$
Therefore, for all $j\ge m$:
$$\|g(\eta)-h'_j\|+\|g(\zeta)-h'_j\|-\|g(\eta)-g(\zeta)\| \le
2k^{-1}\|g(\eta)-g(\zeta)\|.$$ We shall now use
without proof the following simple property of
$N=N_{\tilde\delta_Y}^2$: for any bounded $t$-separated sequence
in $Y$ and any $z\in Y$,
\begin{equation} \label{needed}\liminf_{n\to\infty}(\|y-y_n\|+\|z-y_n\|)
\ge N(\|y-z\|,t).\end{equation}

\noindent Note that for any $\omega\in\Omega$ we have
$$ \|h'_i(\omega)-h'_j(\omega)\|\ge \theta \tau\sigma_k^{-1}-1, \qquad i>j\ge m.$$
Hence, using \eqref{needed}, integrating and using Jensen's
inequality we get $$\liminf_{j\to\infty}
(\|g(\eta)-h'_j\|+\|g(\zeta)-h'_j\|) \ge N(\|g(\eta)-g(
\zeta)\|,\theta\tau\sigma_k^{-1}-1).$$

\noindent Now $\|g(\eta)-g(\zeta)\|\le 8K\tau+1$ and $N(t,1)-t$ is
a decreasing function so
$$ N_Y(8K\tau+1,\theta\tau\sigma_k^{-1}-1)-(8K\tau+1)\le \frac{2}{k}\|g(\eta)-g(\zeta)\|\le 2(8K\tau+1)k^{-1}.$$
Multiplying by $(8K\tau+1)^{-1}$ and letting $\tau$ tend to
$+\infty$ we obtain that
$$\tilde\delta (\frac{\theta}{16K\sigma_k})\le \frac{2}{k}\ \ \
{\rm\ and\ therefore}\ \ \ \tilde\delta (\frac{\theta}{32K\sigma_k})\le \frac{1}{k}$$ or
$$ \|e_1+\cdots+e_k\|_{\tilde\delta_Y}\le 32K\theta^{-1}\sigma_k.$$
\end{proof}

We will end this section by stating two theorems proved
by Kalton in \cite{K6}. Their proofs use, among many other ideas, the results we just explained on the stability of
asymptotic uniform convexity under coarse-Lipschitz embeddings.

\begin{Thm} {\bf(Kalton 2010)} Suppose $1<p<\infty.$  Then

\noindent (i) If $X$ is a Banach space which can be coarse
Lipschitz embedded in $\ell_p,$ then $X$ is linearly isomorphic to
a closed subspace of $\ell_p.$

\noindent (ii) If $X$ is a Banach space which is net equivalent to
a quotient of $\ell_p$ then $X$ is linearly isomorphic to a
quotient of $\ell_p.$

\noindent (iii) If $X$ can be coarse Lipschitz embedded into a
quotient of $\ell_p$ then $X$ is linearly isomorphic to a subspace
of a quotient of $\ell_p.$

\end{Thm}

\noindent {\bf Remark.} On the other hand, Kalton constructed
in \cite{K5} two subspaces (respectively quotients) of $\ell_p$
($1<p\neq 2<\infty$) which are uniformly homeomorphic but not
linearly isomorphic.

\medskip\noindent {\bf Problem 4.} It is not known whether a Banach space which is net equivalent
to a subspace (respectively a quotient) of $c_0$ is linearly
isomorphic to a subspace (respectively a quotient) of $c_0$.

\noindent As we have already seen, a Banach space Lipschitz-isomorphic to a subspace of $c_0$ is linearly isomorphic to a subspace of $c_0$ \cite{GKL}. It is not known if the class of Banach spaces linearly isomorphic to a quotient of $c_0$ is stable under Lipschitz-isomorphisms. However, this question was almost solved by Dutrieux who proved in \cite{Du1} that if a Banach space is Lipschitz-isomorphic to a quotient of $c_0$ and has a dual with the approximation property, then it is linearly isomorphic to a quotient of $c_0$.

\medskip Finally, let us point one last striking
consequence of more general results from \cite{K6}.

\begin{Thm} {\bf(Kalton 2010)} Suppose $1<p,r<\infty$ are such
that $p<\min (r,2)$ or $p>\max(r,2)$, then the space
$(\sum_{n=1}^{\infty}\ell_r^n)_{\ell_p}$ has a unique net
structure.
\end{Thm}

\section{Embeddings of special graphs into Banach spaces}

In this  section we will study special metric graphs or trees that
are of particular importance for the subject. More precisely we
will study the Banach spaces in which they embed. This will allow
us to characterize some linear classes of Banach spaces by a
purely metric condition of the following type: given a metric
space $M$ (generally a graph), what are the Banach  spaces $X$ so
that $M \buildrel {Lip}\over {\hookrightarrow} X$. Or, given a
family $\cal M$ of metric spaces, what are the Banach spaces $X$
for which there is a constant $C\ge 1$ so that for all $M$ in
$\cal M$, $M \buildrel {C}\over {\hookrightarrow}  X$ (i.e. $M$ Lipschitz embeds into $X$ with distorsion at most $C$).

Most of the time these linear classes were already known to be stable under Lipschitz or
coarse-Lipschitz embeddings, when such characterizations were proved. However, we will show one situation (see Corollary \ref{beta}) where this
process yields new results about such stabilities.

The section will be organized by the nature of the linear
properties that can be characterized by such embedding conditions.

\subsection{Embeddings of special metrics spaces and local properties of Banach
spaces}

We already know (see Theorem \ref{Ribelocal}) that local
properties of Banach spaces are preserved under coarse Lipschitz
embeddings. This theorem gave birth to the ``Ribe program'' which
aims at looking for metric invariants that characterize local
properties of Banach spaces. The first occurence of the ``Ribe
program'' is Bourgain's metric characterization of
superreflexivity given in \cite{Bo1}. The metric invariant
discovered by Bourgain is the collection of the hyperbolic dyadic
trees of arbitrarily large height $N$. We denote
$\Delta_{0}=\{\emptyset\}$, the root of the tree. Let
$\Omega_{i}=\{-1,1\}^{i}$,
$\Delta_{N}=\bigcup_{i=0}^{N}\Omega_{i}$ and
$\Delta_\infty=\bigcup_{i=0}^{\infty}\Omega_{i}$. Then we equip
$\Delta_\infty$, and by restriction every $\Delta_N$, with the
hyperbolic distance $\rho$, which is defined as follows. Let $s$
and $s'$ be two elements of $\Delta_\infty$ and let $u\in
\Delta_\infty$ be their greatest common ancestor. We set
$$\rho(s,s')=|s|+|s'|-2|u|=\rho(s,u)+\rho(s',u).$$

Bourgain's characterization is the following:

\begin{Thm}\label{SRBo} {\bf (Bourgain 1986)} Let $X$ be a Banach space. Then $X$ is not superreflexive if and only
if there exists a constant $C \ge 1$ such that for all $N\in
\Ndb$, $(\Delta_N,\rho)\buildrel {C}\over {\hookrightarrow} X$.
\end{Thm}

\noindent {\bf Remark.} It has been proven by Baudier in \cite{Ba2} that this
is also equivalent to the metric embeddability of the infinite
hyperbolic dyadic tree $(\Delta_\infty,\rho)$. It should also be
noted that in \cite{Bo1} and \cite{Ba2}, the embedding constants
are bounded above by a universal constant.

\medskip We also wish to mention that Johnson and Schechtman \cite{JS} recently
characterized the super-reflexivity through the non embeddability
of other graphs such as the ``diamond graphs'' or the Laakso
graphs. We will only give an intuitive description of the diamond
graphs. $D_0$ is made of two connected vertices (therefore at
distance 1), that we shall call $T$ (top) and $B$ (bottom). $D_1$
is a diamond, therefore made of four vertices $T$, $B$, $L$ (left)
and $R$ (right) and four edges : $[B,L],\ [L,T],\ [T,R]$ and
$[R,B]$. Assume $D_N$ is constructed, then $D_{N+1}$ is obtained
by replacing each edge of $D_N$ by a diamond $D_1$. The distance
on $D_{N+1}$ is the path metric of this new discrete graph. The
graph distance on a diamond $D_N$ will be denoted by $d$. The
result is the following.

\begin{Thm}{\bf (Johnson, Schechtman 2009)} Let $X$ be a Banach space.
Then $X$ is not super-reflexive if and only if there is a constant
$C\ge 1$ such that for all $N\in \Ndb$, $(D_N,d) \buildrel
{C}\over {\hookrightarrow} X$.
\end{Thm}

\medskip The metric characterization of the linear type of a
Banach space has been initiated by Enflo in \cite{E2} and
continued by Bourgain, Milman and Wolfson in \cite{BMW}. Let us
first describe a concrete result from \cite{BMW}. For $1\le p\le
2$ and $n\in \Ndb$, $H_p^n$ denotes the set $\{0,1\}^n$
equipped with the metric induced by the $\ell_p$ norm. The metric
space $H_1^n$ is called the Hamming cube. One of their results is
the following.

\begin{Thm} {\bf (Bourgain, Milman, Wolfson 1986)} Let $X$ be a Banach space and $1\le p\le 2$.
Define $p_X$ to be the supremum
 of all $r$'s such that $X$ is of linear type $r$. Then, the following assertions are
 equivalent.

\smallskip (i) $p_X\le p$.

\smallskip (ii) There is a constant $C\ge 1$ such that for all
$n\in \Ndb$, $H_p^n\buildrel {C}\over {\hookrightarrow} X$.

\smallskip In particular, $X$ is of trivial type if and only if $H_1^n\buildrel {C}\over {\hookrightarrow}
X$, for all $n\in \Ndb$ and for some universal constant $C\ge 1$.

\end{Thm}

The fundamental problem of defining a notion of type for metric
spaces is behind this result. Of course we expect such a notion to
coincide with the linear type for Banach spaces and to be stable
under reasonable non linear embedddings. This program was achieved
with the successive definitions of the Enflo type \cite{E2}, the
Bourgain-Milman-Wolfson type \cite{BMW} and finally the scaled
Enflo type introduced by Mendel and Naor in \cite{MN1}. An
even more difficult task was to define the right notion of metric
cotype. This was achieved by Mendel and Naor in \cite{MN2}.
We will not address this subject in this survey, but we strongly
advise the interested reader to study these fundamental papers.

\medskip Let us also describe a simpler metric characterization of the
Banach spaces without (linear) cotype. First let us recall that a
metric space $(M,d)$ is called {\it locally finite} if all its
balls of finite radius are finite. It is of {\it bounded geometry}
if for any $r>0$ there exists $C(r) \in \Ndb$ such that the
cardinal of any ball of radius $r$ is less than $C(r)$. We will
now construct a particular metric space with bounded geometry. For
$k,n\in \Ndb$, denote
$$M_{n,k}=knB_{\ell_\infty^n} \cap n\Zdb^n.$$
Let us enumerate the $M_{n,k}$'s: $M_1,...,M_i,..$ with
$M_i=M_{n_i,k_i}$ in such a way that diam$(M_i)$ is non
decreasing. Note that $\lim_{i} diam(M_i) =+\infty$. Then let
$M_0$ be the disjoint union of the $M_i$'s ($i\ge 1$) and define
on $M_0$ the following distance:

\noindent If $x,y\in M_i$, $d(x,y)=\|x-y\|_\infty$, where $\|\
\|_\infty$ is the $\ell_\infty^{n_i}$ norm.

\noindent If $x\in M_i$ et $y\in M_j$, with $i<j$, set
$d(x,y)=F(j)$, where $F$ is built so that $F$ is increasing and
$$\forall j\ \ \forall i\le j\ \ F(j)\ge \frac12\, diam(M_i).$$
Note that $\lim_i F(i)=+\infty$. We leave it to the reader to
check that $(M_0,d)$ is a metric space with bounded geometry. We
can now state the following.

\begin{Thm}\label{BaLa} Let $X$ be a Banach space. The following assertions
are equivalent.

\smallskip (i) $X$ has a trivial cotype.

\smallskip (ii) $X$ contains uniformly and linearly the
$\ell_\infty^n$'s.

\smallskip (iii) There exists $C\ge 1$ such that for every locally
finite metric space $M$, $M \buildrel {C}\over {\hookrightarrow}
X$.

\smallskip (iv) There exists $C\ge 1$ such that for every metric space with bounded geometry
$M$, $M \buildrel {C}\over {\hookrightarrow} X$.

\smallskip (v) There exists $C\ge 1$ such that  $M_0 \buildrel {C}\over {\hookrightarrow} X$.

\smallskip (vi) There exists $C\ge 1$ such that for every
finite metric space $M$, $M \buildrel {C}\over {\hookrightarrow}
X$.
\end{Thm}

\begin{proof} The equivalence between (i) and (ii) is part of classical
results by Maurey and Pisier \cite{MP}.

(ii)$\Rightarrow$ (iii) is due to Baudier and the second author
\cite{BaLa}.

(iii)$\Rightarrow$ (iv) and (iv)$\Rightarrow$ (v) are trivial.

For any $k,n\in \Ndb$, the space $M_0$ contains the space
$M_{n,k}=knB_{\ell_\infty^n} \cap n\Zdb^n$ which is isometric to
the $\frac1k$-net of $B_{\ell_\infty^n}$: $B_{\ell_\infty^n}\cap
\frac1k \Zdb^n$. But, after rescaling, any finite metric space is
isometric to a subset of $B_{\ell_\infty^n}$, for some $n\in \Ndb$.
Thus, for any finite
metric space $M$ and any $\eps>0$, there exist $k,n\in \Ndb$ so
that $M$ is $(1+\eps)$-equivalent to a subset of $M_{k,n}$. The implication
(v)$\Rightarrow$ (vi) is now clear.

The proof of (vi)$\Rightarrow$ (ii) relies on an argument due to Schechtman \cite{Sc}. So assume that (vi) is satisfied and let us
fix $n\in \Ndb$. Then for any $k\in \Ndb$, there exists a map
$f_k:(\frac{1}{k}\Zdb^n \cap B_{\ell_\infty^n},\|\ \|_\infty) \to X$ such that $f_k(0)=0$
and
$$\forall x,y \in \frac{1}{k}\Zdb^n \cap B_{\ell_\infty^n}\ \ \ \|x-y\|_\infty \le
\|f_k(x)-f_k(y)\|\le K\|x-y\|_\infty.$$ Then we can define a map
$\lambda_k:B_{\ell_\infty^n} \to \frac{1}{k}\Zdb^n \cap B_{\ell_\infty^n}$ such that for all
$x\in B_{\ell_\infty^n}$:
$\|\lambda_k(x)-x\|_\infty=d(x,\frac{1}{k}\Zdb^n \cap B_{\ell_\infty^n})$. We now set
$\varphi_k=f_k\circ \lambda_k$.

\noindent Let $\cal U$ be a non trivial ultrafilter. We define
$\varphi: B_{\ell_\infty^n} \to X_{\cal U}\subseteq X_{\cal U}^{**}$
by $\varphi(x)=(\varphi_k(x))_{\cal U}$. It is easy to check that
$\varphi$ is a Lipschitz embedding. Then it follows from results
by Heinrich and Mankiewicz on weak$^*$-G\^{a}teaux
differentiabilty of Lipschitz maps \cite{HM} that $\ell_\infty^n$
is $K$-isomorphic to a linear subspace of $X_{\cal U}^{**}$.
Finally, using the local reflexivity principle and properties of
the ultra-product, we get that $\ell_\infty^n$ is
$(K+1)$-isomorphic to a linear subspace of $X$.

\end{proof}

\subsection{Embeddings of special graphs and asymptotic structure of Banach
spaces}

We will start this section by considering the countably branching
hyperbolic trees. For a positive integer $N$, $T_N=\bigcup_{i=0}^N
\Ndb^i$, where $\Ndb^0:=\{\emptyset\}$. Then
$T_\infty=\bigcup_{N=1}^\infty T_N$ is the set of all finite
sequences of positive integers. The hyperbolic distance $\rho$ is
defined on $T_\infty$ as follows. Let $s$ and $s'$ be two elements
of $T_\infty$ and let $u\in T_\infty$ be their greatest common
ancestor. We set
$$\rho(s,s')=|s|+|s'|-2|u|=\rho(s,u)+\rho(s',u).$$
The following result, that appeared in \cite{BKL}, is an asymptotic
analogue of Bourgain's characterization of super-reflexivity given in Theorem \ref{SRBo} above.

\begin{Thm}\label{main} Let $X$ be a reflexive Banach space. The following assertions are
equivalent.

(i) There exists $C\ge 1$ such that $T_\infty \buildrel {C}\over
{\hookrightarrow} X$.

(ii) There exists $C\ge 1$ such that for any $N$ in $\Ndb$, $T_N
\buildrel {C}\over {\hookrightarrow} X$.

(iii) $X$ does not admit any equivalent asymptotically uniformly
smooth norm \underline{or} $X$ does not admit any equivalent
asymptotically uniformly convex norm.
\end{Thm}

We will only mention one application of this result.

\begin{Cor}\label{beta} The class of all reflexive Banach spaces that admit both an equivalent
AUS norm and an equivalent AUC norm is stable under coarse
Lipschitz embeddings.
\end{Cor}

\begin{proof} Assume that $X$ coarse Lipschitz embeds in a space $Y$ which is
reflexive, AUS renormable and AUC renormable. First, it follows
from Theorem \ref{reflexive} that $X$ is reflexive. Now the
conclusion is easily derived from Theorem \ref{main}.
\end{proof}

Note that this class coincide with the class of reflexive spaces $X$
such that the Szlenk indices of $X$ and $X^*$ are both equal to the first infinite ordinal $\omega$ (see \cite{GKL2001}).

\medskip\noindent {\bf Problem 5.} We do not know if the class of all Banach spaces that are both AUS renormable and AUC renormable is stable under coarse Lipschitz
embeddings, net equivalences or uniform homeomorphisms.

\medskip\noindent {\bf Problem 6.}  We now present a variant of Problem 2. As we already indicated, we do not know if the class
of reflexive and AUS renormable Banach spaces is stable under
coarse Lipschitz embeddings. The important results by Kalton and
Randrianarivony on the stability of the asymptotic uniform
smoothness are based on the use of particular metric graphs,
namely the graphs $G_k(\Ndb)$ equipped with the distance:
$$\forall \n,\m \in G_k(\Ndb),\ \ d(\n,\m)=|\{j,\ n_j\neq m_j\}|.$$
It seems interesting to try to characterize the Banach
spaces $X$ such that there exists a constant $C\ge 1$ for which
$G_k(\Ndb) \buildrel {C}\over {\hookrightarrow} X$, for all $k\in
\Ndb$. In particular, one may ask whether a reflexive Banach
space which is not AUS renormable always contains the
$G_k(\Ndb)$'s with uniform distortion (the converse implication is
a consequence of Kalton and Randrianarivony's work). A positive
answer would solve Problem 2.

\subsection{Interlaced Kalton's graphs}

Very little is known about the coarse embeddings of metric spaces
into Banach spaces and about coarse embeddings between Banach
spaces (see Definition 3.1. for coarse embeddings). For quite some time it was not even known if a reflexive
Banach space could be universal for separable metric spaces and
coarse embeddings. This was solved negatively by Kalton in
\cite{K0} who showed the following.

\begin{Thm}\label{coarse} {\bf(Kalton 2007)} Let $X$ be a separable Banach space. Assume that $c_0$
coarsely embeds into $X$. Then one of the iterated duals of $X$
has to be non separable. In particular, $X$ cannot be reflexive.
\end{Thm}

The idea of the proof is to consider a new graph metric $\delta$
on $G_k(\Mdb)$, for $\Mdb$ infinite subset of $\Ndb$. We will say
that $\n\neq \m \in G_k(\Mdb)$ are adjacent (or $\delta(\n,\m)=1$)
if they interlace or more precisely if
$$m_1\le n_1\le..\le m_k\le n_k\ \ {\rm or}\ \ n_1\le m_1\le..\le n_k\le m_k.$$

\noindent For simplicity we will only show that $X$ cannot be
reflexive. So let us assume that $X$ is a reflexive Banach space
and fix a non principal ultrafilter $\cal U$ on $\Ndb$. For a
bounded function $f:G_k(\Ndb) \to X$ we define $\partial f:
G_{k-1}(\Ndb)\to X$ by
$$\forall \n\in G_{k-1}(\Ndb)\ \ \ \partial f(\n)= w-\lim_{n_k \in \cal
U}f(n_1,..,n_{k-1},n_k).$$

\noindent Note that for $1\le i \le k$, $\partial^i f$ is a
bounded map from $G_{k-i}(\Ndb)$ into $X$ and that $\partial^k f$
is an element of $X$. Let us first state without proof a series of basic lemmas
about the operation $\partial$.

\begin{Lem}\label{extract} Let $h:G_k(\Ndb) \to \Rdb$ be a bounded map and $\eps>0$. Then
there is an infinite subset $\Mdb$ of $\Ndb$ such that
$$\forall \n \in G_k(\Mdb)\ \ \ |h(\n)-\partial^k h|< \eps.$$
\end{Lem}

\begin{Lem} Let $f:G_k(\Ndb)\to X$ and $g:G_k(\Mdb) \to X^*$ be two bounded
maps. Define $f \otimes g:G_{2k}(\Ndb) \to \Rdb$ by $$(f\otimes
g)(n_1,..,n_{2k})=\langle
f(n_2,n_4,..,n_{2k}),g(n_1,..,n_{2k-1})\rangle.$$ Then
$\partial^2(f\otimes g)=\partial f \otimes \partial g.$
\end{Lem}

\begin{Lem}\label{extract2} Let $f:G_k(\Ndb)\to X$ be a bounded map and $\eps>0$. Then there is
an infinite subset $\Mdb$ of $\Ndb$ such that
$$\forall \n\in G_k(\Mdb)\ \ \|f(\n)\|\le \|\partial^k f\|+
\omega_f(1)+\eps,$$ where $\omega_f$ is the modulus of continuity of $f$.
\end{Lem}

\begin{Lem}\label{extract3} Let $\eps>0$, $X$ be a separable reflexive Banach space and $I$ be an
uncountable set. Assume that for each $i\in I$, $f_i:G_k(\Ndb)\to
X$ is a bounded map. Then there exist $i\neq j \in I$ and an
infinite subset $\Mdb$ of $\Ndb$ such that
$$\forall \n\in G_k(\Mdb)\ \ \|f_i(\n)-f_j(\n)\|\le \omega_{f_i}(1)+\omega_{f_j}(1)+\eps.$$
\end{Lem}

We are now ready for the proof of the theorem. As we will see, the
proof relies on the fact that $c_0$ contains uncountably many
isometric copies of the $G_k(\Ndb)$'s with too many points far
away from each other (which will be in contradiction with Lemma
\ref{extract3}).

\begin{proof}[Proof of Theorem~\ref{coarse}]
Assume $X$ is reflexive and let $h:c_0\to X$ be a map which is
bounded on bounded subsets of $c_0$. Let $(e_k)_{k=1}^\infty$ be
the canonical basis of $c_0$. For an infinite subset $A$ of $\Ndb$
we now define
$$\forall n\in \Ndb\ \ s_A(n)=\sum_{k\le n,\ k\in A} e_k$$
and
$$\forall \n=(n_1,..,n_k)\in G_k(\Ndb)\ \ f_A(\n)=\sum_{i=1}^k s_A(n_i).$$
Then the $h\circ f_A$'s form an uncountable family of bounded maps
from $G_k(\Ndb)$ to $X$. It therefore follows from Lemma
\ref{extract3} that there are two distinct infinite subsets $A$
and $B$ of $\Ndb$ and another infinite subset $\Mdb$ of $\Ndb$ so
that:
$$\forall \n\in G_k(\Mdb)\ \ \|h\circ f_A(\n)-h\circ f_B(\n)\|\le \omega_{h\circ
f_A}(1)+\omega_{h\circ f_B}(1)+1\le 2\omega_h(1)+1.$$ But, since $A\neq B$,
there is $\n\in G_k(\Mdb)$ with $\|f_A(\n)-f_B(\n)\|=k$. By taking
arbitrarily large values of $k$ we deduce that $h$ cannot be a
coarse embedding.
\end{proof}

\noindent {\bf Remarks.}

(1) Similarly, one can show that $h$ cannot be a uniform embedding, by composing $h$ with
the maps $tf_A$ and letting $t$ tend to zero.

\smallskip (2) It is now easy to adapt this proof in order to obtain the stronger
result stated in Theorem \ref{coarse}. Indeed, one just has to
change the definition of the operator $\partial$ as follows. If
$f:G_k(\Ndb)\to X$ is bounded, define $\partial f:G_{k-1}(\Ndb)\to
X^{**}$ by
$$\forall \n\in G_{k-1}(\Ndb)\ \ \ \partial f(\n)= w^*-\lim_{n_k \in \cal
U}f(n_1,..,n_{k-1},n_k).$$
We leave it to the reader to rewrite the argument.

\smallskip (3) On the other hand, Kalton proved in \cite{K-1} that $c_0$
embeds uniformly and coarsely in a Banach space $X$ with the Schur
property. Note that $X$ does not contain any subspace
linearly isomorphic to $c_0$.

\smallskip (4) We will see in the next section that Kalton recently used a similar
operation $\partial$ and the same graph
distance on $G_k(\omega_1)$, where $\omega_1$ is the first
uncountable ordinal (see \cite{K3}) in order to study uniform
embeddings into $\ell_\infty$.

\medskip\noindent {\bf Problem 7.} In view of this result, the
metric graphs $(G_k(\Ndb),\delta)$ are clearly of special
importance. It is a natural question to characterize the
Banach spaces containing the spaces $(G_k(\Ndb),\delta)$ with
uniformly bounded distortion.

\medskip In  \cite{K0} Kalton pushed the idea behind the
proof of Theorem \ref{coarse} much further and introduced the
following abstract notions in order to study the coarse or uniform
embeddings into reflexive Banach spaces.

\smallskip Let $(M,d)$ be a metric space, $\eps>0$ and $\eta\ge 0$.
We say that $M$ has {\it property} $\cal Q(\eps,\eta)$ if for
every $k \in \Ndb$ and every map $f:(G_k(\Ndb),\delta) \to (M,d)$
with $\omega_f(1)\le \eta$ there exists an infinite subset $\Mdb$
of $\Ndb$ such that:
$$d(f(\sigma),f(\tau))\le \eps\ \ \ \sigma<\tau,\ \ \sigma,\tau\in
G_k(\Mdb).$$ Then $\Delta_M(\eps)$ is the supremum of all $\eta
\ge 0$ so that $M$ has {\it property} $\cal Q(\eps,\eta)$ and
Kalton proves the following general statement.

\begin{Thm} Let $M$ be a metric space and $X$ be a reflexive
Banach space.

\noindent (i) If $M$ embeds uniformly into $X$, then
$\Delta_M(\eps)>0$, for all $\eps>0$.

\noindent (ii) If $M$ embeds coarsely into $X$, then $\lim_{\eps
\to +\infty} \Delta_M(\eps)=+\infty$.
\end{Thm}

Let us now turn to the case when our metric space is a Banach
space that we shall denote $E$. Then it easy to see that the
function $\Delta_E$ is linear. We denote $\cal Q_E$ the constant
such that for all $\eps >0$, $\Delta_E(\eps)=\cal Q_E \eps$.
Finally, we say that $E$ has the $\cal Q$-{\it property} if $\cal
Q_E>0$. It follows from Lemma \ref{extract2} that a reflexive
Banach space has the $\cal Q$-property. Thus we have:

\begin{Cor}\label{embedQ} If a Banach space $E$ fails the $\cal Q$-property,
then $E$ does not coarsely embed into a reflexive Banach space and
$B_E$ does not uniformly embed into a reflexive Banach space.
\end{Cor}

The fact that $c_0$ fails the $\cal Q$-property follows from
Theorem \ref{coarse} but it is actually an ingredient of its
proof. Then Kalton continues his study of the links between
reflexivity and the $\cal Q$-property. Let us mention without
proof a few of the many interesting results obtained in \cite{K0}.

\smallskip (1) A non reflexive Banach space with the alternating Banach-Saks
property (in particular with a non trivial type) fails the $\cal
Q$-property.

\smallskip (2) The James space $J$ and its dual fail the $\cal
Q$-property.

\smallskip (3) However, there exists a quasi-reflexive but non
reflexive Banach space with the $\cal Q$-property.

\medskip\noindent {\bf Problem 8.} Is there a converse to
Corollary \ref{embedQ}? More precisely: if $E$ is a separable
Banach space with the $\cal Q$-property, does $B_E$ uniformly
embed into a reflexive Banach space or does $E$ coarsely embed
into a reflexive Banach space? The answer is unknown for the space
constructed in the above statement (3).

\section{Nonseparable spaces}

We collect here a few recent results obtained by Nigel Kalton on
nonseparable Banach spaces together with some related open
problems. All the results presented in this section are taken from
Kalton's paper \cite{K3}. They mainly concern
embeddings of nonseparable Banach spaces into $\ell_\infty$. We
start with a positive result.

\begin{Thm}\label{unc} If $X$ has an unconditional basis and is of density character at most
$c$ (the cardinal of the continuum), then it is Lipschitz embeddable into $\ell_\infty$.
\end{Thm}

\begin{proof}[Sketch of the main ideas in the proof] Assume that the basis is $1-$
unconditional and that it is indexed by the set $\Rdb$ of real
numbers. Denote by $(e_t^*)_{t\in \Rdb}$ the biorthogonal
functionals of the basis. If $x\in X$, we write $x(t)=e^*_t(x)$.
Suppose that $a, b, c\in \Qdb^n$. We write typically, $a=(a_1,
a_2,\dots, a_n)$ and denote by $-a=(-a_1, -a_2,\dots,-a_n)$.

\noindent Define then a subset $U(a, b, c)\subset \Rdb^n$ by
$(t_1, t_2, \dots, t_n)\in U(a, b, c)$ if $b_j<t_j<c_j$ for $j=1,
2, \dots, n$, $t_1< t_2< \dots, t_n$ and

$$
\|\sum_{j=1}^n a_je_{t_j}^*\|_{X^*}\le 1.
$$

\noindent For $t\in \Rdb$ write $t_{+}=\max(t, 0)$ and define
$f_{(a, b, c)}: X \to \Rdb$ by $f_{(a, b, c)}$ is identically $0$
if $U(a, b, c)$ is empty and otherwise

$$
f_{(a, b, c)}(x)=\sup\Big{\{}\sum_{j=1}^n (a_jx(t_j))_{+},\ (t_1, t_2, \dots, t_n)\in U(a, b, c)\Big{\}}.
$$

\noindent Finally define the map

$$
F(x)=\big(f_{(a, b, c)}(x)\big)_{(a, b, c)\in \bigcup_n Q^n}
$$

\noindent It can then be shown that $F$ is a Lipschitz embedding of $X$ into
$\ell_\infty$.
\end{proof}

\noindent {\bf Problem 9.} Let $X$ be reflexive of density $\le c$.
Is $X$ Lipschitz embeddable in $\ell_\infty$?

\bigskip We now proceed with other spaces of density $\le c$. Let $I$ be a set of
cardinality $c$. It is easy to show, using almost disjoint families, that the
space $c_{0}(I)$ is isometric to a subspace of $\ell_\infty/{c_0}$, and it
follows that there is no linear continuous injective map from $\ell_\infty/{c_0}$
into $\ell_\infty$. But by Theorem \ref{unc} above, $c_{0}(I)$ Lipschitz embeds into $\ell_\infty$.
This was shown much earlier \cite{AL} using the space $JL_{\infty}$, and it can also be seen
by applying Theorem VI. 8. 9 in \cite{DGZ} to any separable compact space $K$
with weight $c$ and with some finite derivative empty. Hence the linear
argument does not extend to the non-linear case. However, Kalton showed:

\begin{Thm}\label{kalton} $C[0, \omega_1]$ or $\ell_\infty/{c_0}$ cannot be
uniformly embedded into $\ell_\infty$.
\end{Thm}

Before discussing the main ideas in the proof of this result, we
need some preparation. For $n\ge 0$, let $\Omega_n=\Omega_1^{[n]}$
be the collection of all $n-$subsets of $\Omega_1=[1, \omega_1)$.
For $n=0$, $\Omega_0=\{\emptyset\}$. We write a typical element of
$\Omega_n$ in the form $\alpha=\{\alpha_1, \dots, \alpha_n\}$,
where $\alpha_1<\alpha_2<\cdots< \alpha_n$. If $n\ge 1$, and
$A\subset\Omega_n$, we define $\partial A\subset \omega_1^{[n-1]}$
by $\{\alpha_1, \dots, \alpha_{n-1}\}\in \partial A$ if and only
if $\{\beta:\{\alpha_1, \dots, \alpha_{n-1}, \beta\}\in A\}$ is
uncountable. If $n=1$, this means that $\emptyset\in \partial A$
if and only if $A$ is uncountable.

We will say that $A\subset \Omega_n$ is {\it large} if
$\emptyset\in\partial^n A$. Otherwise $A$ is {\it small}. We will
say that $A\subset\Omega_n$ is {\it very large} if its complement
is small. Then one can show the following Ramsey type result.

\begin{Lem} If $A$ is a very large subset in $\Omega_n$, then there is an
uncountable set $\Theta\subset \Omega_1$ such that $\Theta^{[n]}\subset
A$.
\end{Lem}

We will now make $\Omega_n$ into a graph by declaring
$\alpha\not=\beta$ to be adjacent if they interlace, namely if

$$
\alpha_1\le\beta_1\le\cdots\le\alpha_n\le\beta_n\ \ {\rm or}\ \
\beta_1\le\alpha_1\le\cdots\le \beta_n\le\alpha_n,$$

\noindent and we define $d$ to be the least path metric on
$\Omega_n$, which then becomes a metric space.

We write $\alpha <\beta$ if $\alpha_1<\cdots
<\alpha_n<\beta_1<\cdots\beta_n$. If $\alpha<\beta$, then
$d(\alpha, \beta)=n$ so $\Omega_n$ has diameter $n$.

The next lemma relies on basic properties of the ordered set $\Omega_1$.

\begin{Lem} (i) If $A$ and $B$ are large sets in $\Omega_n$,
then there exist $\alpha\in A$ and $\beta\in B$ so that $\alpha$
and $\beta$ interlace.

\smallskip (ii) If $f$ is a Lipschitz map from $\Omega_n$ into
$\Rdb$, with Lipschitz constant $L$, then there is $\xi\in \Rdb$
so that $\{\alpha,\ |f(\alpha)-\xi|>L/2\}$ is small.

\end{Lem}

It yields:

\begin{Prop}\label{kalton1}

If $f$ is a Lipschitz map from $(\Omega_n, d)\to \ell_\infty$ with
Lipschitz constant $L$, then there is $\xi\in\ell_\infty$ and an
uncountable subset $\Theta$ of $\Omega_1$ so that
$$
\|f(\alpha)-\xi\|\le L/2, \quad \quad \alpha\in\Theta^{[n]}.
$$
\end{Prop}

\begin{proof}[Sketch of the proof of
Theorem \ref{kalton}]

For the case of $C[0, \omega_1]$, let $(x_\mu)_{\mu\le\omega_1}$
be defined by $x_\mu=\chi_{[0, \mu]}$. Assume that $X=C[0,
\omega_1]$ uniformly homeomorphically embeds into $\ell_\infty$ and
let $f: B_X\to \ell_\infty$ be a uniformly homeomorphic embedding.

\noindent Under the notation above, for each $n$, consider the map
$f_n: \Omega_n\to\ell_\infty$ given by
$$
f_n(\alpha)= f(\frac1n\sum_{j=1}^n x_{\alpha_j})
$$

\noindent If $\alpha$ and $\beta$ interlace, then by a telescopic
argument
$$
\|\frac1n\sum_{j=1}^n(x_{\beta_j}-x_{\alpha_j})\|\le \frac2n
$$
\noindent and from the definition of the distance in $\Omega_n$ we
get that $f_n$ has Lipschitz constant $\psi_f(\frac2n)$, where
$\psi_f$ is the modulus of uniform continuity of $f$.

By Proposition \ref{kalton1}, we may pick an uncountable subset
$\Theta_n$ of $\Omega_1$ so that
$$
\|f_n(\alpha)-f_n(\beta)\|\le\psi_f(\frac2n), \quad\quad \alpha, \beta\in
\Theta_n.
$$

\noindent Hence
$$
\|\frac2n\sum_{j=1}^nx_{\beta_j}-\frac2n\sum_{j=1}^nx_{\alpha_j}\|\le
\psi_g(\psi_f(\frac2n)), \quad \quad \alpha, \beta \in \Theta_n
$$

\noindent Pick now $\alpha_1<\alpha_2<\cdots <\alpha_n\in
\Theta_n$. If $\nu>\mu>\alpha_n$, we can find in $\beta_1,..\beta_n$ in $\Theta_n$ so that $\beta_n>\beta_{n-1}>\cdots >\beta_1>\nu$. Then
$$
\|x_\nu-x_\mu\|\le\|\frac1n\sum_{j=1}^nx_{\beta_j}-
\frac1n\sum_{j=1}^nx_{\alpha_j}\|\le\psi_g(\psi_f(\frac2n))
$$
\noindent Thus
$$
\theta(\mu):=\sup_{\sigma>\mu}\|x_\sigma-x_\mu\|\le\psi_g(\psi_f(\frac2n)), \quad \quad \mu>\alpha_n.
$$

\noindent Applying this for every $n$, since
$\lim_{n\to\infty}\psi_g(\psi_f(\frac2n))=0$, we get
$\theta(\mu)=0$ eventually, which is not true.

For $\ell_\infty/c_0$, Theorem \ref{kalton} follows from the case
of $C[0, \omega_1]$ and from the result that $C[0, \omega_1]$ is
linearly isometric to a subspace of $\ell_\infty/c_0$ \cite{Par}.

\end{proof}

Note that Theorem \ref{kalton} implies that there is no quasi-additive Lipschitz projection from $\ell_\infty$ onto $c_0$ (see \cite{BL}). As another application of Theorem \ref{kalton}, Kalton obtains the
following fundamental example.

\begin{Thm}\label{kalton3} There is a (nonseparable) Banach space $Z$ that is not
a uniform retract of its second dual.
\end{Thm}

Before starting on discussing the ideas in the proof, let us
include the following useful lemma of independent interest.

\begin{Lem}\label{kalton4} Let $X$ be a Banach space and let $Q: Y\to X$ be a
quotient mapping. In  order that there is a uniformly continuous selection $f: B_X\to Y$ of
the quotient mapping $Q$, it is sufficient that for some $0<\lambda<1$ there is a uniformly
continuous map $\phi:S_X\to Y$ with $\|Q(\phi(x))-x\|\le\lambda$ for $x\in S_X$.
\end{Lem}

\begin{proof} We extend $\phi$ to $B_X$ to be positively homogeneous and
$\phi$ remains uniformly continuous. Define $g(x)=x-Q(\phi(x))$, so
that $g$ is also positively homogeneous. Then
$\|g(x)\|\le\lambda\|x\|$ and so $\|g^n(x)\|\le\lambda^n\|x\|^n$
for $x\in B_X$. Let $g^0(x)=x$. Let
$$
f(x)=\sum_{n=0}^\infty\phi(g^n(x)).
$$
The series converges uniformly in $x\in B_X$ and so $f$ is
uniformly continuous. Furthermore,
$$
Qf(x)=\sum_{n=0}^\infty(g^n(x)-g^{n+1}(x))=x.
$$
\end{proof}

\begin{proof}[Sketch of the proof of Theorem \ref{kalton3}]

The space $Z$ that we shall consider was constructed by Benyamini in \cite{Ben}. Consider the quotient map
$Q:\ell_\infty\to\ell_\infty/c_{0}$. For each $n$, pick a maximal
set $D_n$ in the interior of $B_{\ell_\infty/c_{0}}$ so that
$\|x-x'\|\ge\frac1n$ for $x, x'\in D_n$ and $x\not= x'$. Then for
each $n$ define a map $h_n: D_n\to B_{\ell_\infty}$ with $Qh_n(x)=
x$ for $x\in D_n$ and denote by $Y_n$ the space $\ell_\infty$ with
the equivalent norm
$$
\|y\|_{Y_n}=\max\big\{\frac1n \|y\|_{\ell_\infty}, \|Qy\|_{\ell_{\infty/c_{0}}}\big\}.
$$

\noindent Note that $Q$ remains a quotient map for the usual norm
on $\ell_{\infty}/c_0$.

Let $Z=(\sum Y_n)_{c_{0}}$ and assume that there is a uniformly
continuous retraction of $B_{Z^{**}}$ onto $B_Z$. Then it follows
that there is a sequence of retractions $g_n:Y^{**}\to Y_n$ which
is equi--uniformly continuous, i.e their moduli of uniform
continuity satisfy
$$
\psi_{g_n}(t)\le\psi(t) \quad 0<t\le 2,\ \ {\rm with}\ \  \lim_{t\to 0}\psi(t)=0.$$
Consider the map $h_n: D_n\to Y_n$. If $x\not= x'\in D_n$, then

$$
\|h_n(x)-h_n(x')\|\le\max\big\{\frac2n, \|x-x'\|\big\}\le 2\|x-x'\|.
$$

\noindent Since $B_{Y_n^{**}}$ is a 1-absolute Lipschitz retract,
there is an extension $f_n: B_{\ell_\infty/c_{0}}\to B_{Y_n^{**}}$
of $h_n$ with $\hbox{Lip}(f_n)\le 2$. Now, if $x\in
B_{\ell_\infty/c_{0}}$, there is $x'\in D_n$ with
$\|x-x'\|<\frac2n$. Thus

$$
\|g_n(f_n(x))-g_n(f_n(x'))\|\le\psi(\frac4n)
$$

\noindent and hence

$$
\|Q(g_n(f_n(x))-x\|\le \psi(\frac4n)+\frac2n.
$$

\noindent Then for $n$ large enough, we have
$$
\psi(\frac4n)+\frac2n<1.
$$

By Lemma \ref{kalton4}, this means that there is a uniformly
continuous selection of the quotient map $Q:
B_{\ell_\infty/c_{0}}\to Y_n$. Thus $B_{\ell_\infty/c_{0}}$
uniformly embeds into $Y_n$ which is isomorphic to $\ell_\infty$
and this is a contradiction with Theorem \ref{kalton}.

\end{proof}

\noindent {\bf Problem 10.} Does there exist a separable, or at least a weakly compactly generated (WCG) Banach space $X$ which is not a Lipschitz retract of its bidual?
\bigskip

We note that if $X$ is Lipschitz embedded in $\ell_\infty$, then
$X$ admits a countable separating family of Lipschitz real valued
functions on $X$. Let us also recall that Bourgain proved in \cite{Bo0} that $\ell_\infty/c_0$ has no equivalent strictly convex norm. This somehow suggests the following problem.
\bigskip

\noindent {\bf Problem 11.} If $X$ is Lipschitz embeddable into
$\ell_\infty$, does $X$ have an equivalent strictly convex norm?
\bigskip

We will finish this section by discussing a few more examples of
nonisomorphic nonseparable spaces that have unique Lipschitz
structure. We will discuss one way of getting many examples by using
the so called pull-back construction in the theory of exact
sequences (see \cite{CCKY}). First, we need some preparation.

A diagram $0\to Y \to X \to Z \to 0$ of Banach spaces and
operators is said to be an exact sequence if the kernel of each
arrow coincides with the image of the preceding one. Hence, in
the diagram above, the second arrow denotes an injection and the
third one is a quotient map. This means, by the open mapping
theorem, that $Y$ is isomorphic to a closed subspace of $X$ and
that the corresponding quotient is isomorphic to $Z$. We say
that $X$ is a twisted sum of $Y$ and $Z$ or an extension of $Y$ by
$Z$.

We say that the exact sequence {\it splits} if the second arrow
$i$ admits a linear retraction (i.e. an arrow $r$ from $X$ into $Y$ so
that $ri=Id_Y$) or equivalently if the third arrow $q$ admits a
linear section, or selection, i.e. if  there is an arrow $s$ from $Z$
into $X$ such that $qs=Id_Z$. This implies that
then $X$ is isomorphic to the direct sum $Y\oplus Z$.

Let $A: U\to Z$ and $B:V\to Z$ be two operators. The {\it
pull-back} of $\{A, B\}$ is the space $PB=\{(u, v): Au=Bv\}\subset
U\times V$ considered with the canonical projections of $U\times
V$ onto $U$ and $V$ respectively.

If $0\to Y\to X\to Z\to 0$ is an exact sequence with quotient map
$q$ and $T:V\to Z$ is an operator and $PB$ denotes the pull-back
of of the couple $\{q, T\}$, then the diagram

$$
\begin{CD}
0 @>>> Y @>>> X @>>> Z @>>>0 \\
@. @.  @ AAA @ AT AA @. \\
0 @>>> Y @>>> PB  @>>> V @>>> 0
\end{CD}
$$
is commutative with exact rows. It follows that the pull-back
sequence splits if and only if the operator $T$ can be lifted to
$X$, i.e. there exists an operator $\tau: V\to X$ such that
$q\tau=T$.

Kalton's  Lemma  \cite{K3} says that if the quotient map in the
first row admits a Lipschitz section then so does the quotient
map in the second row.

We now consider the following pull-back diagram.

$$
\begin{CD}
0 @>>> c_0 @>>> JL_\infty @>>> c_0(I) @>>>0 \\
@. @AI{c_0}AA  @ A\hat{T}AA @ AT AA @. \\
0 @>>> c_0 @>>> JL_2  @>>> \ell_2(I) @>>> 0 \\
@. @AI{c_0}AA   @ A\hat{S}AA @ASAA @.\\
0 @>>> c_0 @>>> CC @>>> \ell_\infty @>>> 0
\end{CD}
$$

\noindent In this diagram, the operator $T$ is the inclusion map
and the operator $S$ is the Rosenthal quotient map \cite{Ros}.
Therefore we get that the space $JL_2$ of Johnson and
Lindenstrauss is an example of a non WCG space that it is
Lipschitz isomorphic to $\ell_2(I)\oplus c_0$ and Lipschitz embeds
into $\ell_\infty$. However, it is known that it does not linearly
isomorphically embeds into $\ell_\infty$ (see \cite{JL}).

The space $CC$ in the third row is then an example of a space that
is Lipschitz homeomorphic to $\ell_\infty\oplus c_0$, but is not
linearly isomorphic to it. Note also that $CC$ is linearly
isomorphic to a subspace of $\ell_\infty$ because $JL_\infty$ linearly embeds into $\ell_\infty$ and $CC$ is a subspace of $JL_\infty\oplus \ell_\infty$ (see \cite{CC} for details).

Similarly, Kalton obtained for instance that any nonseparable WCG
space that contains an isomorphic copy of $c_0$ fails to have
unique Lipschitz structure.

\medskip
\noindent{\bf Problem 12.} Does every reflexive (superreflexive)
space have unique Lipschitz structure?

\medskip

\noindent{\bf Problem 13.} Does $\ell_\infty$ have unique Lipschitz
structure?

\medskip\noindent {\bf Remark:} It turns out that the undecidability of the continuum hypothesis (and of related
statements) casts a shadow on the Lipschitz classification of non-separable spaces. Let $I$ be a set of
cardinality $c$. We assume that the
space $c_{0}(I)$ is Lipschitz- isomorphic to a WCG Banach space $X$. Does it follow that $X$ is
linearly isomorphic to   $c_{0}(I)$? It follows from \cite{GKL} that the answer is positive if
$c<\aleph_{\omega_0}$, where $\aleph_{\omega_0}$ denotes the $\omega_0$-th cardinal.
On the other hand, it follows from \cite{BMar} and \cite{Mar} that the answer is negative if $c >\aleph_{\omega_0}$ (see also \cite{ACGJM}
for a related result on the failure of a non-separable Sobczyk theorem at the density $\aleph_{\omega_0}$).
Since (ZFC) does not decide if $c$ is strictly less or strictly more than $\aleph_{\omega_0}$,
the above question is undecidable in (ZFC).

\section{Coarse embeddings into Banach spaces and geometric group
theory}

Two topological spaces $M$ and $N$ are homotopically equivalent if there exist
continuous maps $f: M\rightarrow N$ and $g: N\rightarrow M$ such that
$f\circ g$ and $g\circ f$ are both homotopic to the identity map on the space
on which they operate. For instance, any topological {\it vector} space is
homotopically equivalent to a point (in other words, is contractible).

We consider the case where $M$ and $N$ are real compact manifolds. In order to
understand their geometry, it is of course important to find quantities which are invariant
under homotopy equivalence. A basic example is the signature of the manifold, which can
be defined as follows when the dimension $n=4k$ is divisible by 4. If $d_j$ denotes the
exterior derivative acting on the differential forms of degree $j$, then $F_k=Ker (d_{2k})$ is the vector space
of closed forms of degree $2k$ and $E_k=Im(d_{2k-1})$ is its subspace of exact forms of degree $2k$.
If $\omega_1$ and $\omega_2$ are two forms in $F_k$ and if we define $Q$ by

$$Q(\omega_1, \omega_2)=\int_M \omega_1\wedge\omega_2$$

\noindent then $Q$ is a bilinear symmetric form, and an easy computation shows that the value of $Q$
depends only upon the classes of $\omega_1$ and $\omega_2$ in the finite-dimensional quotient space $H_{k}(M)=F_k/E_k$. Therefore $Q$ defines a quadratic form on $H_{k}(M)$, and its signature is called the signature of the manifold $M$. This signature is invariant under homotopy equivalence. Actually, a theorem of Novikov asserts (for simply connected manifolds) that the signature is the only homotopy invariant which can be computed in terms of quantities called Pontryagin polynomials.

We now recall the basics of geometric group theory. Let $G$ be a finitely generated group, and $S$ be a finite set generating $G$. We can equip $G$ with the word distance $d_S$ associated to $S$, as follows: if $\Vert g\Vert_S$ is the minimal length of a word written with elements of $S$ and $S^{-1}$ representing $g$, then we define the left-invariant distance $d_S$ on $G$ by $d_S(g_1, g_2)=\Vert g_1^{-1}g_2\Vert_S$. If for instance $G$ is the group $\Zdb^n$ and $S$ is its generating set consisting of the unit vector basis, then $d_S$ coincide with the distance induced by the $\ell_1$ norm on ${\Rdb}^n$. It is easily checked that if $S$ and $S'$ are two finite generating sets, then the identity map is a coarse Lipschitz isomorphism between the metric spaces $(G, d_S)$ and $(G, d_{S'})$.
A property of finitely generated groups is said to be {\sl geometric} if it depends only upon the space $(G, d_S)$ up to coarse Lipschitz equivalence, Many natural properties of groups (such as amenability, hyperbolicity, or being virtually nilpotent, i.e. containing a nilpotent subgroup of finite index) turn out to be geometric.

We denote again by $M$ a real compact manifold. Novikov's conjecture asserts that certain ``higher signatures" are homotopy invariants. We would have to introduce
several highly non-trivial concepts before providing a precise statement of this conjecture and this
is not the purpose of this survey. However it is easy to describe the link between Novikov's conjecture and geometric group theory. Indeed, let $\pi_1(M)$ be the first homotopy group of the manifold $M$. Mikhail Gromov conjectured a link between the geometry of the finitely generated group $\pi_1(M)$ and the Novikov conjecture, and this conjecture was confirmed by Yu: Novikov's conjecture (and even the stronger coarse Baum-Connes conjecture) holds true if the
group $\pi_1(M)$ equipped with the word distance coarsely embeds into the Hilbert space \cite{Y}. This important result was later generalized by Kasparov and Yu \cite{KY} who showed that the Hilbert space can be replaced in this statement by any super-reflexive space. This is indeed a generalization since the spaces $\ell_p$ with $p>2$ do not coarsely embed into the Hilbert space \cite{JR}. Note that conversely, it is an open question to know if the separable Hilbert space coarsely embeds into every infinite-dimensional Banach space, in other words if coarse embedding into $\ell_2$ is the strongest possible property of that kind. A word of warning is needed here: what we call ``coarse embedding" is often called (after Gromov) ``uniform embedding" in the context of differential geometry. In this survey however, the word uniform bears another meaning.

Let us recall that a metric space $E$ is called locally finite if every ball of $E$ is finite, and it has {\sl bounded geometry} if for any $r>0$, the cardinality of subsets of diameter less than $r$ is uniformly bounded. The left invariance of the distance $d_S$ shows that any finitely generated group has bounded geometry. Therefore the question occurs to decide which finitely generated group, and more generally which space with bounded geometry coarsely embeds into a super-reflexive space. For instance, could it be that every space with bounded geometry coarsely embeds into a super-reflexive space?

This question is now negatively answered. A first example of a locally finite space which does not coarsely embed into the Hilbert space is obtained in \cite{DGLY}, using in particular a construction of Enflo \cite{En}. Then Gromov shows \cite{Gr} through a random approach the existence of finitely generated groups $G$ such that the metric space $(G, d_S)$ coarsely contains a sequence of expanders $E_i$ - that is, a sequence of graphs such that the first positive eigenvalue of the Laplacian is uniformly bounded below - such that the girth of $E_i$, namely the length of the shortest closed curve, increases to infinity. As shown in \cite{LLR}, embeddings of expanders into the Hilbert space have maximal distortion. It follows that $G$ cannot be coarsely embedded into a super-reflexive space, and since such a group can be realized as an homotopy group (see \cite{Gr2}) it cuts short hopes to prove the full Novikov conjecture through the coarse embedding approach (see also \cite{HLS}).

On the other hand, the problem remains open to decide which metric spaces coarsely embed into Banach spaces of given regularity. Any metric space with bounded geometry coarsely embeds into a reflexive space \cite{BG}. This result is widely generalized in \cite{K0} where it is shown that every stable metric space (where ``stable" means that the order of limits can be permuted in $\lim_k \lim_n d(x_k, y_n)$ each time all limits exist) can be coarsely embedded into a reflexive space. This is indeed extending \cite{BG} since every metric space whose balls are compact, and thus every locally finite metric space, is stable. Moreover, it follows from Theorem \ref{BaLa} that any locally finite metric space Lipschitz embeds into the following very simple reflexive space: $(\sum_{n=1}^\infty \ell_\infty^n)_{\ell_2}$ (note that this space is both AUC and AUS). On the other hand, Theorem \ref{coarse} states that $c_0$ does not coarsely embed into a reflexive space, nor into a stable metric space (by the above, or directly by \cite{Ray}). Note that the important stable Banach space $L^1$ coarsely embeds into the Hilbert space \cite{JR},\cite{Ran}.

As seen before, coarse embeddings of special graphs bear important consequences on the non-linear geometry of Banach spaces. Sequences of expanders shed light on the geometry of groups and its applications to homotopy invariants. It is plausible that such expanders could provide several interesting examples in geometry of Banach spaces.

\medskip
For the record, we recall the

\medskip\noindent {\bf Problem 14.} Does $\ell_2$ coarsely embed into every infinite dimensional Banach space?

\section{Lipschitz-free spaces and their applications}

Let $M$ be a pointed metric space, that is, a metric space equipped with a
distinguished point denoted $0$. The space $\Lip_0(M)$ is the space of real-valued
Lipschitz functions on $M$ which vanish at $0$. Let $\mathcal{F}(M)$ be the natural predual of $\Lip_0(M)$, whose
$w^*$-topology coincide on the unit ball of $\Lip_0(M)$ with the
pointwise convergence on $M$.  The Dirac map $\delta : M \to \mathcal{F}
 (M)$ defined by $\langle g, \delta(x)\rangle =g(x)$ is an isometric
embedding from $M$ to a subset of $\mathcal{F}(M)$ which
generates a dense linear subspace. This predual $\mathcal{F}(M)$ is called in
\cite{GK} the {\sl Lipschitz-free space over $M$}. When $M$ is separable,
$\mathcal{F}(M)$ is separable as well since $\delta(M)$ spans a dense subspace.
Although Lipschitz-free spaces over separable metric spaces constitute a class of
separable Banach spaces which are easy to define, the structure of these spaces
is very poorly understood to this day. Improving our understanding of this class
is a fascinating research program. Note that if we identify (through the Dirac map)
a metric space $M$ with a subset of $\mathcal{F}(M)$, any Lipschitz map from $M$ to a
metric space $N$ extends to a continuous linear map from $\mathcal{F}(M)$ to
$\mathcal{F}(N)$. So Lipschitz maps become linear, but of course the complexity is shifted
from the map to the free space: this may explain why the structure of Lipschitz-free spaces
is not easy to analyze. A first example is provided by the real line, whose free space is isometric
to $L_1$. Actually, metric spaces $M$ whose free space is isometric to a subspace of $L_1$ are characterized in \cite{God} as subsets of metric trees equipped with the least path metric. On the other hand, the free space of the plane $\Rdb^2$ does not embed isomorphically into $L^1$ \cite{NS}.

Banach spaces $X$ are in particular pointed metric spaces (pick the origin as distinguished point)
and we can apply the previous construction. Note that the isometric embedding
$\delta:X\to \mathcal{F}(X)$ is of course non linear
since there exist Lipschitz functions on $X$ which are not affine.

This Dirac map has a linear left
inverse $\beta: \mathcal{F}(X) \to X$ which is the quotient map such
that $x^*(\beta(\mu))= \langle x^*, \mu \rangle$ for all $x^* \in X^*$; in
other words, $\beta$ is the extension to $\mathcal{F}(X)$ of the
barycenter map.  This setting provides canonical examples of
Lipschitz-isomorphic spaces.  Indeed, if we let $Z_X= Ker (\beta)$,
it follows easily from $\beta \delta = Id_X$ that $Z_X \oplus X =
\mathcal{G}(X)$ is Lipschitz-isomorphic to $\mathcal{F}(X)$.

Following \cite{GK}, let us say that a Banach space $X$ has the {\it
lifting property} if there is a continuous linear map $R:X\to \mathcal{F}(X)$ such
that $\beta R=Id_X$, or equivalently, if for $Y$ and $Z$ Banach spaces and
$S: Z\rightarrow Y$ and $T:X\rightarrow Y$ continuous linear maps, the existence of a
Lipschitz map $\mathcal{L}$ such that $T=S \mathcal{L}$ implies the
existence of a continuous linear operator $L$ such that $T=SL$.  A
diagram-chasing argument shows that $\mathcal{G}(X)$ is linearly isomorphic to
$\mathcal{F}(X)$ if and only if $X$ has the lifting property \cite{GK}.
It turns out that all non-separable reflexive spaces, and also the
spaces $\ell_\infty (\mathbb{N})$ and $c_0(\Gamma)$ when $\Gamma$ is
uncountable, fail the lifting property and this provides canonical
examples of pairs of Lipschitz but not linearly isomorphic spaces.

On the other hand, the following result is proved in \cite{GK}:

\begin{Thm}\label{lifting}: Every separable Banach space $X$ has the lifting
property.

\end{Thm}

 \begin{proof} We will actually give two proofs. First, one can pick a Gaussian measure $\gamma$
whose support is dense in $X$ and use the result that $(\delta
* \gamma)$ is G{\^{a}}teaux-differentiable.  Then in the above notation
$R=(\delta * \gamma)' (0)$ satisfies $Id_X =\beta R$.

The second proof is essentially self-contained. It consists into replacing the Gaussian measure by a cube measure, and this will be useful later.
It underlines the simple fact that being separable is equivalent to being ``compact-generated". Again, we use differentiation, but only in the directions which are normal to the faces of the cube.

Let $(x_i)_{i\geq 1}$ be a linearly independent sequence of vectors in $X$
such that $$\overline{span}\,[(x_i)_{i\geq 1}]=X$$ and $\Vert x_i\Vert=2^{-i}$ for all $i$.
Let $H=[0,1]^{\Ndb}$ be the Hilbert cube and $H_n=[0,1]^{{\Ndb}_n}$ be the copy of the
Hilbert cube where the factor of rank $n$ is omitted; that is, ${\Ndb}_n={\Ndb}\backslash\{n\}$.
We denote by $\lambda$ (resp. $\lambda_n$) the natural probability measure on $H$ (resp. $H_n$)
obtained by taking the product of the Lebesgue measure on each factor.

We denote $E={span}\,[(x_i)_{i\geq 1}]$  and $R:E\to \mathcal{F}(X)$
the unique linear map which satisfies for all $n\geq 1$ et all $f\in Lip_0(X)$

$$R(x_n)(f)=\int_{H_n}\big[f(x_n+\displaystyle\sum_{j=1, j\not=n}^{\infty}t_j x_j)-f(\displaystyle\sum_{j=1, j\not=n}^{\infty}t_j x_j)\big]d\lambda_n(t)$$

\vskip 5 mm
Pick $f\in \Lip_0(X)$. If the function $f$ is  G{\^{a}}teaux-differentiable, Fubini's theorem shows that for all $x\in E$

$$R(x)(f)=\int_{H}<\{\nabla f\}(\displaystyle\sum_{j=1}^{\infty}t_j x_j), x> d\lambda(t) $$

Thus $\vert R(x)(f)\vert\leq\Vert x\Vert\,\Vert f\Vert_L$ in this case. But since $X$ is separable, any  $f\in \Lip_0(X)$ is a uniform limit of a sequence $(f_j)$ of G{\^{a}}teaux-differentiable functions such that $\Vert f_j\Vert_L\leq\Vert f \Vert_L$. It follows that $$\Vert R\Vert\leq 1.$$

\vskip 5 mm
We may now extend $R$ to a linear map $\overline{R}: X\to \mathcal{F}(X)$ such that $\Vert\overline{R}\Vert=1$ and it is clear that $\overline{R}(x)(x^*)=x^*(x)$ for all $x\in X$ and all $x^*\in X^*$. Hence $\beta \overline{R}=Id_X$.

\end{proof}

The above proof follows \cite{GK}. We refer to \cite{G1} for an elementary approach along the same lines, which uses only finite-dimensional arguments and is accessible at the undergraduate level.

The lifting property for separable spaces forbids the existence of a {\it separable} Banach space $X$ such that $\cal F(X)$ and $\cal G(X)$ are not linearly isomorphic, but on the other hand it
shows that if there exists an isometric
embedding from a separable Banach space $X$ into a Banach space $Y$, then $Y$ contains a
linear subspace which is isometric to $X$.
Indeed a theorem due to Figiel \cite{F} states that if $J:X \to Y$ is an
isometric embedding  such that $J(0)=0$ and $\overline{span}
[J(X)]=Y$ then there is a linear quotient map with $\| Q\|=1$ and
$QJ = Id_X$, and then the lifting property provides a linear contractive map $R$
such that $QR=Id_X$, and this map $R$ is a linear isometric embedding.  We note
that $P=RQ$ is a contractive projection from $Y$ onto $R(X)$. This remark is developed
further in \cite{G2} where it is shown that the existence of
a non-linear isometric embedding from $X$ into $Y$ is a very restrictive condition on
the couple $(X,Y)$.

Nigel Kalton constructed the proper frame for showing the gap which separates
H\"{o}lder maps from Lipschitz ones \cite{K-1}.  If $(X, \| \ \|)$ is a Banach space
and $\omega:[0, + \infty) \to [0,+ \infty)$ is a subadditive function
such that $\lim_{t \to 0} \omega(t) = \omega(0) =0$ and $\omega(t) =t$ if $t \geq
1$, then the space $\Lip_\omega(X)$ of ($\omega \circ d$)-Lipschitz functions
on $X$ which vanish at $0$ has a natural predual denoted
$\mathcal{F}_\omega (X)$, and the barycentric map $\beta_\omega:\mathcal{F}_\omega
(X) \to X$ (whose adjoint is the canonical embedding from $X^*$ to
$\mathcal{F}_\omega(X)$) is still a linear quotient map such that
$\beta_\omega \delta = Id_X$.  However, the Dirac map $\delta: X \to
\mathcal{F}_\omega(X)$ is now uniformly continuous with modulus $\omega$ -
e.g. $\alpha$- H\"{o}lder when $\omega(t) = \max (t^\alpha, t)$ with
$0<\alpha<1$. Uniformly continuous functions fail the differentiability
properties that Lipschitz functions enjoy, and thus one can
expect that this part of the theory is more ``distant'' from the
linear theory than the Lipschitz one.  It is indeed so, and \cite[Theorem 4.6]{K-1},
 reads as follows.

\begin{Thm}\label{Holder}
If $\omega$ satisfies $\lim_{t\to 0} \frac{\omega(t)}{t}= \infty$, then
$\mathcal{F}_\omega (X)$ is a Schur space - that is, weakly convergent
sequences in $\mathcal{F}_\omega (X)$ are norm convergent.
\end{Thm}

It follows from Theorem \ref{Holder} that the uniform analogue of the lifting
property fails unless $X$ has the (quite restrictive) Schur
property.  Moreover, $\mathcal{F}_\omega (X)$ is $(3\omega)$-uniformly
homeomorphic to $[X \oplus Ker (\beta_\omega)]$ and as soon as $\lim_{t
\to 0} \frac{\omega(t)}{t} =0$ and $X$ fails the Schur property we obtain
canonical pairs of uniformly (even H\"{o}lder) homeomorphic separable Banach spaces
which are not linearly isomorphic.  We refer to \cite{R1,JLS} for
other examples of such pairs.

Along with H\"{o}lder maps between Banach spaces, one may as
well consider Lipschitz maps between quasi-Banach spaces, and this is
done in \cite{AK} where similar methods provide examples of separable
quasi-Banach spaces which are Lipschitz but not linearly isomorphic.

We now observe that the proof (with cube measures) of Theorem \ref{lifting} provides
the existence of compact metric spaces whose free space fails the approximation property (in short, A.P.).
This has been observed in \cite{GO}.

\begin{Thm}\label{Ozawa}
There exists a compact metric space $K$ whose free space $\mathcal{F}(K)$ fails the
approximation property.
\end{Thm}

\begin{proof} We use the notation of the proof of Theorem \ref{lifting}. Let $C$ be the closed convex hull of the sequence $(x_i)_{i\geq 1}$, and let $K=2C$. It is easily seen that the map $R$ takes its values in the closed subspace $\mathcal{F}(K)$ of $\mathcal{F}(X)$, and so does $\overline{R}$. It follows that $X$ is isometric to a 1-complemented subspace of $\mathcal{F}(K)$, through the projection $\overline{R}Q$. If this construction is applied to a Banach space $X$ which fails A.P. , then $\mathcal{F}(K)$ fails A. P. as well since A. P. is carried to complemented subspaces.

\end{proof}

\noindent {\bf Problem 15.} Let $X$ be a separable Banach space, and $Y$ a Banach space which is Lipschitz-isomorphic to $X$. Does it follow that $Y$ is linearly isomorphic to $X$?

This question amounts to know if every separable Banach space is determined by its metric structure.
It is open for instance if $X=\ell_1$ or if $X=C(K)$ with $K$ a countable compact metric space, unless $C(K)$ is isomorphic to $c_0$. Note that by the above the answer to this question is negative if we drop the separability assumption, or if we replace Lipschitz by H\"{o}lder, or if we replace Banach by quasi-Banach.

\medskip\noindent {\bf Problem 16.} Is the Lipschitz-free space ${\mathcal F}(\ell_1)$ over $\ell_1$ complemented in its bidual?

A motivation for this question is that if ${\mathcal F}(\ell_1)$ is complemented in its bidual, it follows that every space $X$ which is Lipschitz-isomorphic to $\ell_1$ is complemented in its bidual, and then  \cite[Corollary 7.7]{BL} shows that $X$ is linearly isomorphic to $\ell_1$.

 \medskip\noindent {\bf Problem 17.} Theorem \ref{Ozawa} leads to the question of knowing for which compact spaces $K$ the space $\mathcal{F}(K)$ has A. P. or its metric version M. A. P. So far, very little is known on this topic, which is related with the existence of linear extension operators for Lipschitz functions (see \cite{Bor}, \cite{GO}).

 \medskip\noindent {\bf Problem 18.} Let $M$ be an arbitrary uniformly discrete metric space, that is, there exists $\theta>0$
such that $d(x,y)\geq\theta$ for all $x\not=y$ in $M$.
Does $\mathcal{F}(M)$ have the B. A. P? Note that A. P. holds by (\cite[Proposition 4.4]{K-1}).
A positive answer to this question would imply that every separable Banach space $X$
is approximable, that is, the identity $Id_X$ is the pointwise limit of an equi-uniformly
continuous sequence of maps with relatively compact range. By \cite[Theorem 4.6]{K4},
it is indeed so for $X$ and $X^*$ when $X^*$ is separable, and in particular every separable reflexive space is approximable. On the other hand, a negative answer to this question would provide an equivalent norm on $\ell_1$ failing M. A. P. and this would solve a famous
problem in approximation theory, by providing the first example of a dual space - namely, $\ell_{\infty}$ equipped with the corresponding dual norm - with A. P. (and even B. A. P.) but failing M. A. P.

\bigskip\noindent {\bf Aknowledgements.} The authors would like to thank F. Baudier, M. Kraus and V. Montesinos for useful discussions.

\end{document}